\documentclass[a4paper,11pt,oneside]{article}

\makeindex
\usepackage{hyperref}
\usepackage{graphicx}
\usepackage{enumerate}
\hypersetup{colorlinks, linkcolor=blue, urlcolor=red, filecolor=green, citecolor=blue}
\usepackage{amsmath}
\usepackage{amsthm}
\usepackage{amssymb}
\usepackage{authblk}
\usepackage{enumerate}
\usepackage{fullpage}
\usepackage{esint}
\usepackage{mathtools}

\DeclarePairedDelimiter\ceil{\lceil}{\rceil}
\DeclarePairedDelimiter\floor{\lfloor}{\rfloor}

\theoremstyle{plain}
\newtheorem{theorem}{Theorem}[section]
\newtheorem{corollary}[theorem]{Corollary}
\newtheorem{lemma}[theorem]{Lemma}
\newtheorem{proposition}[theorem]{Proposition}

\theoremstyle{definition}
\newtheorem{definition}[theorem]{Definition}

\theoremstyle{remark}
\newtheorem{remark}{Remark}

\newcommand\supp{\mathop{\operatorname{supp}}}
\newcommand\dist{\operatorname{dist}}

\title{Asymptotes of macroscopic observables in Gibbs measures of general interacting particle systems}
\author{David Padilla-Garza}

\begin{document}

\maketitle

\begin{abstract}
This paper studies the Gibbs measure of an interacting particle system with a general interaction kernel at various temperature regimes. We are particularly interested in fine features of the convergence to the mean-field density as the number of particles tends to infinity. Our main results are concentration bounds, and estimates on the Laplace transform of fluctuations. The main technique is a regularization procedure for general interaction kernels, based on an associated parabolic flow. 
\end{abstract}

\section{Introduction and motivation}

We will study an interacting particle system on the torus, in which the particles interact via a repulsive kernel, and are confined by an external potential. This is modeled by the Hamiltonian
\begin{equation}
\label{eq:hamiltonian}
    \mathcal{H}_{N}(X_{N}) = \sum_{i \neq j} g(x_{i} - x_{j}) + N \sum_{i=1}^{N} V(x_{i}),
\end{equation}
where $g: \mathbf{T}^{d} \to \mathbf{R}$ is the repulsive kernel, $V: \mathbf{T}^{d} \to \mathbf{R}$ is the confining potential, $N$ is the number of particles, $X_{N} = (x_{1}, ... , x_{N}) \in \mathbf{T}^{d \times N}$, and the $d-$dimensional torus of length $T$,$\mathbf{T}^{d}$ is defined as $\mathbf{T}^{d}:= \mathbb{R}^{d} / \sim$, where the equivalence relation $\sim$ is defined as $x \sim y$ if and only if $x - y \in T\mathbf{Z}^{d}$. We are interested in the behavior for large but finite $N$. A system modeled by equation \eqref{eq:hamiltonian} will be called an interacting particle system, or many-particle system. \\

Interacting particle systems are also widely studied in $\mathbf{R}^{d}$, not only $\mathbf{T}^{d}$. In this case, a frequent form of the repulsive interaction $g$ is given by
\begin{equation}
\label{Coulombcase}
\begin{cases}
g(x)=\frac{c_{d}}{|x|^{d-2}} \text{ if } d \geq 3, \\
g(x)=-c_{2} \log(|x|) \text{ if } d = 2,
\end{cases}
\end{equation}
where $c_{d}$ is such that
\begin{equation}
    -\Delta g  = \delta_{0}
\end{equation}
for all $d \geq 2$ \footnote{There is an explicit formula for $c_{d}$, but it is not relevant for our purpose.}. We will refer to this setting as the Euclidean Coulomb case. Another frequent form of $g$ is given by 
\begin{equation}\label{Rieszcase}
    g(x) = \frac{c_{d,s}}{|x|^{d-2s}},
\end{equation}
with $0<s<\min\{ \frac{d}{2},1 \}$,  $d \geq 1$, and $c_{d,s}$ such that
\begin{equation}
    (- \Delta)^{s} g = \delta_{0}.\footnote{There is an explicit formula for $c_{d,s}$, but it is not relevant for our purpose.}
\end{equation}
We will refer to this setting as the Euclidean Riesz case. \\

Motivated by this, we may consider Coulomb and Riesz kernels on the torus as well. In this case, we work with $g: \mathbf{T}^{d} \to \mathbf{R}$ such that
\begin{equation}
    (-\Delta g)^{s} = \delta_{0}.
\end{equation}
We will refer to this setting as a periodic Riesz case for $0<s < \min\{\frac{d}{2}, 1\}$, and as a periodic Coulomb case for $s=1$. \\

Finally, another widely studied interaction in $1-$dimensional Euclidean space is 
\begin{equation}
    g(x) = - \log |x|.
\end{equation}
We will refer to this setting as the $1d$ log case.

At zero temperature, the particles will arrange themselves into a configuration that minimizes the Hamiltonian. However, we may also study interacting particle systems at positive temperature. In this case, the position of the particles is a random variable in $\mathbf{T}^{d \times N}$. The density of the random variable is given by the Gibbs measure,
\begin{equation}
   \mathrm d \mathbf{P}_{N, \beta} (X_{N}) = \frac{1}{Z^{V}_{N, \beta}} \exp \left( - \beta \mathcal{H}_{N}(X_{N}) \right) \, \mathrm d X_{N},
\end{equation}
where
\begin{equation}
\label{eq:partfunc}
    Z^{V}_{N, \beta} = \int_{\mathbf{T}^{d \times N}} \exp \left( - \beta \mathcal{H}_{N}(X_{N}) \right) \, \mathrm d X_{N}
\end{equation}
is the partition function, and $\beta >0$ is the inverse temperature. We will be particularly interested in settings in which $\beta$ is not fixed, but might depend on $N$. \\

Coulomb gases are related to vortices in the Ginzburg-Landau model of superconductivity, Fekete points in approximation theory, and vortices in Bose-Einstein condensates \cite{serfaty2017microscopic}. Coulomb and $1-$d log gases are also related to Random Matrix Theory, since they can correspond to Ginibre ensembles, $\beta-$ensembles, and the eigenvalues of real, complex, or quaternionic Gaussian matrices. A (largely non-exhaustive) list of papers at the intersection of Coulomb gases and Random Matrix Theory is \cite{johansson1998fluctuations, bourgade2012bulk, bourgade2014universality, bourgade2014local, lambert2021mesoscopic, hardy2021clt, lambert2019quantitative}.  Coulomb gases are also connected to complex geometry and mathematical physics \cite{berman2014determinantal, rougerie2015incompressibility}.    

The study of a general interacting gas has received attention in the past \cite{georgii2011gibbs,ruelle1967variational,ruelle1968statistical} and continues to draw attention in current research \cite{garcia2019large, lambert2021poisson, chafai2014first, nguyen2022mean, rosenzweig2023global}. Apart from being a natural generalization of Coulomb and Riesz gasses, interacting particle systems with a general pair-wise interaction are also linked to AI and machine learning, more specifically to neural networks. One of the uses of neural networks is to accurately represent high-dimensional functions: given a high-dimensional function $f$, one goal of neural networks is to represent $f$ as
\begin{equation}
    f (x) := \lim_{n \to \infty} \frac{1}{n} \sum_{i=1}^{n} {\varphi}_{i} (x, \theta_{i}),
\end{equation}
where ${\varphi}_{i} (x, \theta_{i})$ are given functions, which depend parameters $\theta_{i}$. Stochastic Gradient Descent (SGD) is one of the most frequent algorithms used to determine the parameters $\theta_{i}$. However, even in spite of its frequent use, very few rigorous convergence results for SGD existed until recently. The approach in \cite{rotskoff2018neural,rotskoff2018parameters,rotskoff2022trainability,wen2024coupling} to derive such a rigorous convergence result is to model the evolution of the parameters $\theta_{i}$ under the SGD as a particle system with an evolution given by an SPDE. The authors are able to show a convergence rate of $O(n^{-1})$ to the mean-field limit by modelling parameters in this way. An additional benefit of this approach is that the energy landscape for the empirical measure is convex. The interaction between the particles in this model depends on on the functions ${\varphi}_{i}$, and on the error between the measurements and the approximating function. It is generally not a Coulomb or Riesz interaction, however. This approach is linked to particle systems for large $d$ as well, since the function $f$ is high dimensional. \\

For a large number of particles $N$, the energy of the system is well-described by the mean-field limit of the Hamiltonian, denoted by $\mathcal{E}_{V}$, and acting on probability measures as
\begin{equation}\label{meanfieldlimit}
    \mathcal{E}_{V} (\mu) =  \mathcal{E} (\mu) + \int_{\mathbf{T}^{d}} V 
    \mathrm d \mu, 
\end{equation}
where 
\begin{equation}
    \mathcal{E} (\mu) = \int_{\mathbf{T}^{d} \times \mathbf{T}^{d}} g(x-y) \, \mathrm d  \mu \otimes \mu(x,y)
\end{equation}
denotes the self-interaction of a measure. Under mild assumptions on $V$, the mean-field limit has a unique minimizer in the set of probability measures, denoted by $\mu^{V}_{\infty}$ and called the \emph{equilibrium measure}:
\begin{equation}
    \mu^{V}_{\infty}:= \underset{\mu \in \mathcal{P}(\mathbf{T}^{d})}{\rm argmin} \, \mathcal{E}_{V}(\mu),
\end{equation}
where $\mathcal{P}(\mathbf{T}^{d})$ denotes the set of probability measures on $\mathbf{T}^{d}$.

At a macroscopic level, the system is characterized by the empirical measure ${\rm emp}_{N}$, defined as 
\begin{equation}
 {\rm emp}_{N} (X_{N}) = \frac{1}{N} \sum_{i=1}^{N} \delta_{x_{i}}. 
\end{equation}
In order to ease notation, we will often write $ {\rm emp}_{N}$ instead of $ {\rm emp}_{N} (X_{N})$. 

A classical result (see, for example, \cite{serfaty2015coulomb}) in Coulomb gases is that $\frac{1}{N^{2}} \mathcal{H}_{N}$ converges ($\Gamma-$converges) to the mean-field limit $\mathcal{E}_{V}$. In particular, this implies convergence of minima and minimizers of $\mathcal{H}_{N}$ to minima and minimizers of $\mathcal{E}_{V}$.   

As long as the temperature is not too large, it is natural to expect that the behaviour of a general interacting particle system under the Gibbs measure is similar to the behaviour in the $0$ temperature case. This intuition is confirmed in \cite{garcia2019concentration, chafai2014first}, which proves that, as long as $\frac{1}{N} \ll \beta$, the empirical measure converges a.s. under the Gibbs measure, to the equilibrium measure. Furthermore, \cite{garcia2019concentration, chafai2014first} quantify rare events by means of an LDP. 

For large temperatures ($\beta = \frac{1}{N}$), the behavior of the system is significantly different. In this case, the effect of the entropy (temperature) is relevant at leading order. The system is no longer well described by the mean-field limit, but by a modified functional which incorporates the effect of entropy (temperature). This functional is denoted $ \mathcal{E}_{V}^{\theta}$, and acts on a measure $\mu$ as
\begin{equation}\label{thermallimit}
     \mathcal{E}_{V}^{\theta} (\mu) =  \mathcal{E}_{V} (\mu) + \frac{1}{ \theta} {\rm ent}[\mu], 
\end{equation}
with
\begin{equation}
{\rm ent}[\mu ]=
    \begin{cases}
    \int_{\mathbf{T}^{d}}  \log  ({\mathrm d \mu})  \, \mathrm d {\mu} \quad  \text{ if } \mu \text{ is absolutely continuous w.r.t. Lebesgue measure} \\
    \infty\quad  \text{ o.w.}
    \end{cases}
\end{equation} 
Under mild assumptions $ \mathcal{E}_{V}^{\theta}$ has a unique minimizer in $\mathcal{P}(\mathbf{T}^{d})$, denoted $\mu^{V}_{\theta}$ and called the \emph{thermal equilibrium measure}:
\begin{equation}
\label{def:theqmeas}
    \mu^{V}_{\theta}:= \underset{\mu \in \mathcal{P}(\mathbf{T}^{d})}{\rm argmin}\,  \mathcal{E}_{V}^{\theta} (\mu).
\end{equation}

In the large temperature regime $\beta = \frac{1}{N}$, the effect of the temperature is large enough that the particles do not converge to the equilibrium measure. Instead, the empirical measure converges a.s. under the Gibbs measure to the thermal equilibrium measure. Similarly to other regimes, it is possible to understand rare events in this system using an LDP \cite{garcia2019large}. Furthermore, the thermal equilibrium measure provides a better approximation to the system than the equilibrium measure, even if $\frac{1}{N} \ll \beta$ \cite{padilla2023concentration, garcia2022generalized}.  

The goal of this paper is to obtain a finer, more quantitative understanding of the convergence of the empirical measure to either the equilibrium or thermal equilibrium measure in the context of a general interaction. In this way, we extend many results in \cite{garcia2022generalized} to general interactions. 

One way to obtain such a finer, quantitative understanding is by concentration bounds. \cite{garcia2022generalized} proves upper bounds for the probability that a certain norm of the difference between the (thermal) equilibrium measure is larger than a certain threshold. This approach is not available in this setting, since, unlike the case of Coulomb or Riesz-type gases, there is no natural choice of norm. Instead, we deal directly with a test function: we look at the difference between the empirical and (thermal) equilibrium measures, integrated against a test function. This quantity will be called a fluctuation.

We will also be interested in the Laplace transform of fluctuations. Apart from serving to provide a finer and more precise understanding of the convergence to the mean-field density, the Laplace transform is an interesting object in its own right. We will obtain an upper bound for the Laplace transform in terms of a simple quantity that is valid as long as the test function is regular enough and also show that in some cases, such a bound is optimal.  

We will be interested in working with a singular and repulsive interaction kernel and a confining potential that are as general as possible. In this case, we obtain the leading-order terms in concentration bounds and the Laplace transform. Under additional hypotheses, we also obtain information about the rate of convergence. Generally speaking, this rate of convergence is given by an inverse power of $N$, and improves as a function of 3 factors:
\begin{itemize}
    \item[1.] The regularizing effect of $g$ (the mildness of the singularity). 

    \item[2.] The regularity of the (thermal) equilibrium measure.

    \item[3.] The regularity of the test function.  
\end{itemize}

The main technical contribution needed to achieve these estimates is a regularization procedure for general interaction kernels, which allows us to approximate the energy of a discrete probability measure by the energy of a continuous one. This regularization procedure is based on a parabolic equation associated with the interaction kernel. In the case of the Coulomb interaction, this parabolic equation is the heat equation. 

For the rest of the paper, we commit an abuse of notation by not distinguishing between a measure and its density. 

\section{Main results}

We now proceed to state the main results of this work. As mentioned in the introduction, part of the results will concern only the leading-order terms in systems with a very general class of interaction kernels. This class will be called \emph{weakly admissible}.  

\begin{definition}
 We call an interaction kernel $g$ \emph{weakly admissible} if $g$ satisfies:   
\begin{itemize}
    \item[1.] Symmetry.
    \begin{equation}
        g(x) = g(-x).
    \end{equation}
    
    \item[2.] Integrability.
    \begin{equation}
        g \in L^{1}(\mathbf{T}^{d}).
    \end{equation}
    
    \item[3.] Positive definiteness. For any $\mu \in TV(\mathbf{T}^{d})$,
    \begin{equation}
        \mathcal{E}(\mu) \geq 0,
    \end{equation}
    and $\mathcal{E}(\mu) = 0 \implies \mu =0$, where $TV(\mathbf{T}^{d})$ denotes the space of signed measures of bounded variation.
\end{itemize}
\end{definition}

There is little hope of obtaining more quantitative estimates about convergence for weakly admissible kernels since they must apply, in particular, to any Riesz kernel. Hence, most of the work will be about obtaining more quantitative estimates for kernels satisfying additional assumptions. This class will be called \emph{admissible}. 

\begin{definition}
\label{def:admissible}
We call an interactions kernel $g:\mathbf{T}^{d} \to \mathbf{R}$ \emph{admissible} if $g$ satisfies:
\begin{itemize}
    \item[1.] $g$ is symmetric, i.e. $g(x)=g(-x)$.

    \item[2.] There exist $\gamma, C >0$ such that, for $m \neq 0$,
    \begin{equation}
        \widehat{g}(m) \leq C |m|^{- \gamma},
    \end{equation}
    where $\widehat\cdot$ denotes the Fourier transform. \footnote{Note that the Fourier transform is well-defined since any admissible $g$ is integrable by assumption 5.}
    \item[3.] $\widehat{g}(m) >0$ for all $m \in \mathbf{Z}^{d}$, and there exist $\lambda, c >0$ such that for $m \neq 0$,
    \begin{equation}
        \widehat{g}(m) \geq c |m|^{- \lambda}.
    \end{equation}

    \item[4.] There exist $\epsilon>0$ and $C < \infty$ such that for all $z \in \mathbf{T}^{d}$ and $0<s < \epsilon$,
    \begin{equation}
        p(z, s) \geq -C,
    \end{equation}
    see equation \eqref{eq:kernelp} for the definition of $p(z,s)$. 

    \item[5.] There exists $C>0$ such that either
    \begin{itemize}
        \item[5 a)]  \begin{equation}
        |g(x)| \leq C \frac{1}{|x|^{s}}, \quad |\nabla g(x)| \leq C \frac{1}{|x|^{s+1}}
    \end{equation}
    for some $s \in (0,d)$, or 
        \item[5 b)] \begin{equation}
        |g(x)| \leq C |\log(|x|)|, \quad |\nabla g(x)| \leq C \frac{1}{|x|}.\footnote{With an abuse of notation, we are writing $|x|$ to denote $\dist(x,0)$.}
    \end{equation}
    \end{itemize}
   
\end{itemize}
\end{definition}

\begin{remark}[Notes on assumptions]
    Assumption 4 is technical, and so the exact definition is postponed until Section \ref{sect:regularization}. Consequently, assumption 4 is not easy to verify. However, it is not very restrictive: as $s$ tends to $0$, $p(\cdot, s)$ converges to a Dirac delta at $0$, so having $p(\cdot, s)$ unbounded below for arbitrarily small $s$ would be pathological behavior. Assumption 4 holds for common kernels, like Coulomb, and Riesz. It also holds for Riesz-type kernels (see \cite{garcia2022generalized,nguyen2022mean}).   Assumption 4 also holds if $h := \widehat{\left(\frac{1}{\widehat{g}}\right)}$ satisfies that $h >0$ a.e. and $h(x) \leq C |x|^{s}$ for some $s<2$. More generally, assumption 4 also holds for any kernel $g$ such that the inverse of $u \mapsto g \ast u$ satisfies a comparison principle, see Proposition \ref{prop:sufficient}. We do not know of any interaction for which assumption 4 can be shown \emph{not} to hold.

    Assumption 5 is only needed in the proof of the lower concentration bound, Theorem \ref{teo:conclow} item a). 

    Note that assumption 3 implies that any admissible kernel $g$ is positive definite, and assumption 5 implies that any admissible kernel $g$ is integrable. Therefore any admissible kernel is weakly admissible. 

\end{remark}

We will also work with a general class of confining potentials, which we now define. 

\begin{definition}
We call a confining potential $V$ \emph{admissible} if $V$ satisfies:
\begin{itemize}
    \item[1.] $V$ is l.s.c.
    \item[2.] $V \in L^{1}(\mathbf{T}^{d})$. 
\end{itemize}
\end{definition}
Note that, as a consequence of assumption 1, $V$ is bounded below, and therefore $\int_{\mathbf{T}^{d}} V \, \mathrm d \mu$ is well-defined for any probability measure $\mu$. 

As mentioned in the introduction, the rate of convergence in admissible kernels will depend on the regularity of the (thermal) equilibrium measure, and of the test function. This regularity is no Sobolev norm but is defined through the Fourier coefficients. 

\begin{definition}
    Given $\alpha \geq 0$, we define the fractional Sobolev seminorm $W^{\alpha}$ as 
    \begin{equation}
        \|\varphi\|_{W^{\alpha}} := \sum_{m \in \mathbf{Z}^{d}} \left| \widehat{\varphi}(m)|m|^{\alpha} \right|.
    \end{equation}

    We define the space $W^{\alpha}(\mathbf{T}^{d})$ as the space of distributions with finite $W^{\alpha}$ seminorm.
\end{definition}

\begin{remark}
\label{rem:Holder}
    Note that for any distribution $\varphi$ and $\alpha >0$, $\|\varphi\|_{C^{\alpha}} \leq C \|\varphi\|_{W^{\alpha}}$, where $\|\varphi\|_{C^{\alpha}}$ denotes the Holder $\alpha-$seminorm. This can be easily verified since, for any $x, y \in \mathbf{T}^{d}$,
    \begin{equation}
        \begin{split}
            \frac{\varphi(x) - \varphi(y)}{|x-y|^{\alpha}} & \leq \sum_{m \in \mathbf{Z}^{d}} \frac{\widehat{\varphi}(m) \left| \exp\left( \frac{2 \pi m}{T} \cdot x \right) - \exp\left( \frac{2 \pi m}{T} \cdot y \right) \right|}{|x-y|^{\alpha}}\\
            &\leq C \sum_{m \in \mathbf{Z}^{d}} \left| \widehat{\varphi}(m)|m|^{\alpha} \right|.
        \end{split}
    \end{equation}
    In particular, $W^{\alpha}(\mathbf{T}^{d})$ is contained in the space of continuous functions for any $\alpha>0$.   
\end{remark}

To obtain concentration bounds, and estimates on the Laplace transform of fluctuations, we will need preliminary estimates on the partition function. These estimates will come in two kinds: comparisons with the equilibrium measure, and comparisons with the thermal equilibrium measure. 

\begin{lemma}
\label{lem:partfunc}
Assume that $V$ and $g$ are admissible and let $\alpha \geq \gamma$. Define $\theta = N \beta$. Then
\begin{itemize}
\item[a)] \begin{equation}
    -\mathcal{E}_{V}^{\theta}(\mu^{V}_{\theta}) + \frac{1}{N} \mathcal{E}(\mu^{V}_{\theta}) \leq \frac{\log Z^{V}_{N, \beta}}{N^{2} \beta} \leq -\mathcal{E}_{V}^{\theta}(\mu^{V}_{\theta}) +C\left( 1 +\|\mu^{V}_{\theta}\|_{W^{\alpha - \gamma}} \right)N^{p^{*}},
\end{equation}
where 
\begin{equation}
    p^{*} = p^{*} (\alpha, \lambda, \gamma) := \max \left\{ -\frac{\gamma}{d}, \left( \frac{\lambda}{\alpha} - \frac{\lambda d}{\alpha \gamma} -1 \right)^{-1} \right\}.
\end{equation}

    \item[b)] If $\lim_{N \to \infty} \theta = \infty$, then
    \begin{equation}
    \begin{split}
    &-\mathcal{E}_{V}(\mu^{V}_{\infty}) - \frac{1}{\theta} {\rm ent}[\mu^{V}_{\infty}] + \frac{1}{N} \mathcal{E}(\mu^{V}_{\infty}) \leq \frac{\log Z^{V}_{N, \beta}}{N^{2} \beta} \leq\\
    &-\mathcal{E}_{V}(\mu^{V}_{\infty})+C\left( 1+\|\mu^{V}_{\infty}\|_{W^{\alpha - \gamma}} \right)N^{p^{*}}  - \frac{1}{\theta} \left( \log |\Sigma| + o_{N}(1)  \right).
    \end{split}
\end{equation}
\end{itemize}

\end{lemma}

The proof is found in Section \ref{sec:asymps}

We now move on to the main results in this work: concentration bounds and estimates on the Laplace transform of fluctuations. We will look at fluctuations of the empirical measure with respect to both the equilibrium and thermal equilibrium measures. The reason for looking at fluctuations with respect to the equilibrium measure is clear: it is the a.s. limit of the empirical measure as long as $\frac{1}{N} \ll \beta$. Our reason for looking at fluctuations with respect to the thermal equilibrium measure is that, on the one hand, in the large temperature regime ($\beta = \frac{1}{N}$), it is the a.s. limit of the empirical measure. Furthermore, in the Coulomb and Riesz cases, it is known to provide a better approximation than the equilibrium measure to the empirical measure \cite{padilla2023concentration,garcia2022generalized}. 

\begin{definition}
Given a continuous function $f: \mathbf{T}^{d} \to \mathbf{R}$, we define the random variables
\begin{equation}
\begin{split}
    {\rm Fluct}_{\infty}[f]  &:= \int_{\mathbf{T}^{d}} f d \left( {\rm emp}_{N} - \mu^{V}_{\infty} \right),\\
    {\rm Fluct}_{\theta}[f]  &:= \int_{\mathbf{T}^{d}} f d \left( {\rm emp}_{N} - \mu^{V}_{\theta} \right).
\end{split}    
\end{equation}
\end{definition}

We now introduce an operator that will play a major role in the rest of the paper. 
\begin{definition}
Given $\alpha \in \mathbf{R}$ we define the operator $h_{\alpha}$ on the space of tempered distributions, in terms of Fourier coefficients as:
\begin{itemize}
    \item If either $\alpha>0$ or $m \neq 0$,
    \begin{equation}
    \widehat{h_{\alpha}(\mu)} (m) := \widehat{g}^{\alpha}\widehat{\mu}(m).
\end{equation}

    \item If $\alpha \leq 0$ and $m=0$ 
\end{itemize}
\begin{equation}
    \widehat{h_{\alpha}(\mu)} (0):= 0.
\end{equation}

We denote $h(\mu) = h_{1}(\mu)$. Note that $h(\mu) = g \ast \mu$. Note that this operator is well defined for an admissible $g$, since for smooth $f$,  $\sum_{m \in \mathbf{Z}^{d}} |\widehat{f}|^{k}(m) < \infty$ for all $k>0$.  
\end{definition}

We now state the main results of this paper, beginning with concentration bounds. Unless otherwise stated, $C$ will denote a generic constant that depends only on $V$ and $g$. 

\begin{theorem}[Concentration bounds]
\label{teo:concup}
Let $f \in W^{\kappa}(\mathbf{T}^{d})$, let $r \in \mathbf{R}^{+}$. Define $\theta = N \beta$. Then
    If $g$ is weakly admissible and $\lim_{N \to \infty} \theta = \infty$, then
    \begin{equation}
    \label{eq:upboundweak}
        \limsup_{N \to \infty} \frac{-1}{N^{2} \beta} \log \left( \mathbf{P}_{N, \beta} \left( \left| {\rm Fluct}_{\infty}[f]  \right| \geq r \right) \right) \leq r^{2} \left\| h_{-\frac{1}{2}} (f) \right\|_{L^{2}}^{2}.
    \end{equation}

    If $g$ is admissible, then for any $\alpha >0$,
    \begin{equation}
    \label{eq:conctherm}
    \mathbf{P}_{N, \beta} \left( \left| {\rm Fluct}_{\theta}[f]  \right| \geq r \right) \leq \exp \Bigg( -N^{2} \beta \Bigg[ \frac{ \left(r- C N^{q^{*}}\|f\|_{W^{\kappa}} \right )^{2} }{\left\| h_{-\frac{1}{2}}(f) \right\|^{2}_{L^{2}}} - \left(   C  \left( 1+ \|\mu^{V}_{\theta}\|_{W^{\alpha - \lambda}} \right)N^{q^{*}}\right)  \Bigg] \Bigg),
    \end{equation}
        where
    \begin{equation}
       q^{*}(\alpha, \kappa, \lambda, \gamma) := p^{*}(\min\{\alpha, \kappa\}, \lambda, \gamma).
    \end{equation}
    If, in addition, $\lim_{N \to \infty} \theta = \infty$, then for any $\alpha >0$,
    \begin{equation}
    \label{eq:conceq}
    \begin{split}
    &\mathbf{P}_{N, \beta} \left( \left| {\rm Fluct}_{\infty}[f]  \right| \geq r \right)\\
    &\leq \exp \Bigg( -N^{2} \beta \Bigg[ \frac{ \left(r- C N^{q^{*}}\|f\|_{W^{\kappa}} \right )^{2} }{\left\| h_{-\frac{1}{2}}(f) \right\|^{2}_{L^{2}}} -    C  \left( 1+ \|\mu^{V}_{\infty}\|_{W^{\alpha - \lambda}} \right)N^{q^{*}}  \Bigg]  \\
    &\ \ \ \ + N\left( {\rm ent}[\mu^{V}_{\infty}] - \log \left|\Sigma\right| + o_{N}(1) \right) \Bigg),
    \end{split}
    \end{equation}
    where $o(1)$ is independent of $f$.   
\end{theorem}

The proof is found in Section \ref{sec:conccentration}.

A natural question is whether the upper bounds proved in Theorem \ref{teo:concup} are optimal. We address this question of optimality in the next theorem, by proving lower bounds on the corresponding probability. Theorem \ref{teo:conclow} shows that, under mild additional conditions, the concentration bounds of Theorem \ref{teo:concup} are optimal to leading order. 

\begin{theorem}[Concentration lower bounds]
\label{teo:conclow}
Let $f \in W^{\kappa}(\mathbf{T}^{d})$, let $r \in \mathbf{R}^{+}$. Define $\theta = N \beta$. Then
\begin{itemize}
    \item[a)] If $f$ satisfies:
    \begin{itemize}
        \item[i)] ${\rm supp}(h_{-1}(f)) \subset \Sigma$,
        \item[ii)] There exists $r_{0} \in \mathbf{R}^{+}$ such that for any $r \in (0,r_{0})$, $\mu^{V}_{\infty} + r \frac{h_{-1}(f)}{\| h_{-\frac{1}{2}}(f) \|^{2}_{L^{2}}} \geq 0$,
    \end{itemize}
    $g$ is weakly admissible, and $\lim_{N \to \infty} \theta = \infty$, then for any $r \in (0,r_{0})$,
    \begin{equation}
    \label{eq:loboundweak}
        \liminf_{N \to \infty} \frac{-1}{N^{2} \beta} \log \left( \mathbf{P}_{N, \beta} \left( \left| {\rm Fluct}_{\infty}[f]  \right| \geq r \right) \right) \geq r^{2} \left\| h_{-\frac{1}{2}} (f) \right\|_{L^{2}}^{2}.
    \end{equation}
    If, in addition, $g$ is admissible, then for any $m > m^{*} := \max \left\{ p^{*}, \frac{s - d}{d} \right\}$ ($s$ is taken to be $0$ in the case of assumption 5 b) of Definition \ref{def:admissible}),
    \begin{equation}
    \label{eq:lowbound1}
    \begin{split}
    &\mathbf{P}_{N, \beta} \left( \left| {\rm Fluct}_{\infty}[f]  \right| \geq r \right) \\
    \geq  &\exp \Bigg( -N^{2}\beta \left[ \left(r + C N^{- \alpha p }  \|f\|_{W^{\alpha}}\right)^{2} \| h_{-\frac{1}{2}}(f) \|_{L^{2}}^{-2} +C(1+\|\mu^{V}_{\infty}\|_{W^{\alpha - \gamma}} ) N^{m}  \right] \\
    &\ \ \ \ \ \ \ \ - N \left(  R(f,r) + o(1) \right) \Bigg),
    \end{split}
    \end{equation}
    where 
    \begin{equation}
        R(f,r) := \mathrm{ent}\left[\mu^{V}_{\infty} + r \frac{h_{-1}(f)}{\| h_{-\frac{1}{2}}(f) \|^{2}_{L^{2}}} \Bigg| \frac{1}{|\Sigma|}\mathbf{1}_{\Sigma}\right].
    \end{equation}
      \item[b)] If there exists $r_{0} \in \mathbf{R}^{+}$ such that for any $r \in (0,r_{0})$, $\mu^{V}_{\infty} + r \frac{h_{-1}(f)}{\| h_{-\frac{1}{2}}(f) \|^{2}_{L^{2}}} \geq 0$, and $g$ is admissible, then for any $m > m^{*} := \max \left\{ p^{*}, \frac{s - d}{d} \right\}$ ($s$ is taken to be $0$ in the case of assumption 5 b) of Definition \ref{def:admissible}), then for any $r \in (0,r_{0})$,
    \begin{equation}
    \label{eq:lowbound2}
    \begin{split}
    &\mathbf{P}_{N, \beta} \left( \left| {\rm Fluct}_{\theta}[f]  \right| \geq r \right) \\
    \geq  &\exp \Bigg( -N^{2}\beta \left[ \left(r + C N^{- \alpha p }  \|f\|_{W^{\alpha}}\right)^{2} \| h_{-\frac{1}{2}}(f) \|_{L^{2}}^{-2} +C(1+\|\mu^{V}_{\theta}\|_{W^{\alpha - \gamma}} ) N^{m}  \right] \\
    &\ \ \ \ \ \ \ \ - N \left(    \widetilde{R}(f,r) + o(1) \right) \Bigg),
    \end{split}
    \end{equation}
     where 
    \begin{equation}
        \widetilde{R}(f,r) := \mathrm{ent}\left[\mu^{V}_{\theta} + r \frac{h_{-1}(f)}{\| h_{-\frac{1}{2}}(f) \|^{2}_{L^{2}}} \Bigg| \mu^{V}_{\theta} \right].
    \end{equation}
\end{itemize}
\end{theorem}

The proof is found in Section \ref{sec:conccentration}.

\begin{remark}
     At low temperature ($\beta$ constant or tending to $0$ slowly), we expect the equilibrium measure $\mu^{V}_{\infty}$ and the thermal equilibrium measure $\mu^{V}_{\theta}$ to be very similar (this has been proved in the Coulomb setting, see \cite{armstrong2022thermal}) and consequently ${\rm Fluct}_{\theta}[f]$ and ${\rm Fluct}_{\infty}[f]$ to be similar (their difference converging to $0$ rapidly with $N$). However, for large temperature ($\beta$ tending to $0$ rapidly) we expect $\mu^{V}_{\infty}$ and $\mu^{V}_{\theta}$ to be different, hence ${\rm Fluct}_{\theta}[f]$ and ${\rm Fluct}_{\infty}[f]$ to be less similar (their difference converging to $0$ slowly with $N$, or not converging to $0$ at all). This is the reason for stating results for both ${\rm Fluct}_{\infty}[f]$ and ${\rm Fluct}_{\theta}[f]$. 
\end{remark}

We now state the other main result of this paper, concerning the Laplace transform of fluctuations. 

\begin{theorem}[Laplace transform]
\label{teo:laplace}
Let $f \in W^{\kappa}(\mathbf{T}^{d})$. Define $\theta = N \beta$. Then
\begin{itemize}
    \item[a)] 
    If $g$ is weakly admissible and $\lim_{N \to \infty} \theta = \infty$ then
    \begin{equation}
    \label{eq:uplaplaceweak}
        \limsup_{N \to \infty} \frac{1}{N^{2} \beta } \log \left( \mathbb{E}_{\mathbf{P}_{N, \beta}} \left( \exp \left( N^{2} \beta r {\rm Fluct}_{\infty}[f]\right) \right) \right) \leq \frac{r^{2}}{4} \|h_{-\frac{1}{2}}(f)\|_{L^{2}}^{2}. 
    \end{equation}
    If $g$ is admissible, then for any $\alpha>0$,
    \begin{equation}
    \begin{split}
   &\log \left( \mathbb{E}_{\mathbf{P}_{N, \beta}} \left( \exp \left( N^{2} \beta r {\rm Fluct}_{\theta}[f]\right) \right) \right) \leq \\
   &N^{2} \beta \left[\frac{ r^{2} \left\| h_{-\frac{1}{2}}(f) \right\|^{2}_{L^{2}}}{4} -  C N^{q^{*}}\|f\|_{W^{\kappa}} r - C \left(1+ \|\mu^{V}_{\theta}\|_{W^{\alpha - \lambda}} \right) N^{q^{*}}  \right].   
   \end{split}
    \end{equation}
    If $\lim_{N \to \infty} \theta = \infty$, then for any $\alpha>0$,
    \begin{equation}
    \begin{split}
   &\log \left( \mathbb{E}_{\mathbf{P}_{N, \beta}} \left( \exp \left( N^{2} \beta r {\rm Fluct}_{\infty}[f]\right) \right) \right) \\
   & \leq  N^{2} \beta \left[\frac{ r^{2} \left\| h_{-\frac{1}{2}}(f) \right\|^{2}_{L^{2}}}{4} -  C N^{q^{*}}\|f\|_{W^{\kappa}} r - C  \left( 1+ \|\mu^{V}_{\infty}\|_{W^{\alpha - \lambda}} \right)N^{q^{*}} \right] \\
   & \ \ \ \ \  + N\left( {\rm ent}[\mu^{V}_{\infty}] - \log \left|\Sigma\right| + o_{N}(1) \right),   
   \end{split}
    \end{equation}
    where $o(1)$ is independent of $f$.

    \item[b)] If $f$ satisfies:
    \begin{itemize}
        \item[i)] ${\rm supp}(h_{-1}(f)) \subset \Sigma$
        \item[ii)] $\mu_{\infty}^{V} - r\frac{ h_{-1}(f)}{2} \geq 0$,
    \end{itemize}
    $g$ is weakly admissible, and $\lim_{N \to \infty} \theta = \infty$, then
    \begin{equation}
    \label{eq:lowlaplaceweak}
        \limsup_{N \to \infty} \frac{1}{N^{2} \beta } \log \left( \mathbb{E}_{\mathbf{P}_{N, \beta}} \left( \exp \left( N^{2} \beta r {\rm Fluct}_{\infty}[f]\right) \right) \right) \geq \frac{r^{2}}{4} \|h_{-\frac{1}{2}}(f)\|_{L^{2}}^{2}. 
    \end{equation}
    If, in addition, $g$ is admissible, then
    \begin{equation}
    \begin{split}
   \log \left(  \mathbb{E}_{\mathbf{P}_{N, \beta}} \left( \exp \left( N^{2} \beta r {\rm Fluct}_{\infty}[f]\right) \right) \right) \geq   N^{2} \beta \Bigg[ \frac{r^{2}}{4} \| h_{-\frac{1}{2}}(f) \|_{L^{2}}^{2} -\\
    -C\left( 1+ \|\mu^{V}_{\infty}\|_{W^{\alpha - \gamma}} \right) N^{p^{*}}  + \frac{1}{\theta} \left( {\rm ent}[\mu^{V}_{\infty}] - \log |\Sigma| + o_{N}(1)  \right)  \Bigg], 
   \end{split}
    \end{equation}
    where $o(1)$ is independent of $f$. 
\end{itemize}
\end{theorem}

The proof is found in Section \ref{sec:laplace}.

\begin{remark}
    Using Jensen's inequality, we have that
    \begin{equation}
        \begin{split}
             {\rm ent}\left[\mu^{V}_{\infty}\right] - \log |\Sigma| &\geq {\rm ent}\left[ \frac{1}{|\Sigma|} \mathbf{1}_{\Sigma}\right] - \log |\Sigma|\\
             &=0.
        \end{split}
    \end{equation}
\end{remark}

\begin{remark}
    Theorems \ref{teo:concup}, \ref{teo:conclow}, and \ref{teo:laplace} are stated in terms of $\left\|h_{-\frac{1}{2}}(f)\right\|_{L^{2}}$ since they give the optimal estimate. Identifying the term that makes the upper bounds sharp (to leading order) is part of the goal of Theorems \ref{teo:concup}, \ref{teo:conclow}, and \ref{teo:laplace}. However, the reader may want to relate this to more tractable quantities, and have concentration bounds in terms of more common norms. This can be achieved by a comparison with $\dot{H}^{s}$ semi-norms. In the Coulomb case, $\left\|h_{-\frac{1}{2}}(f)\right\|_{L^{2}}$ is (up to a universal constant) equal to $\left\|\nabla f\right\|_{L^{2}} = \left\| f \right\|_{\dot{H}^{1}} $. In the Riesz case, $\left\|h_{-\frac{1}{2}}(f)\right\|_{L^{2}}$ is (up to a universal constant) equal to $\left\| (- \Delta)^{\frac{s}{2}} f\right\|_{L^{2}}= \left\| f \right\|_{\dot{H}^{\frac{s}{2}}} $. In the admissible case, assumptions 2 and 3 of Definition \ref{def:admissible} imply that $c \left\| f \right\|_{\dot{H}^{\frac{\lambda}{2}}} \leq \left\|h_{-\frac{1}{2}}(f)\right\|_{L^{2}} \leq C  \left\| f \right\|_{\dot{H}^{\frac{\gamma}{2}}} $. 
\end{remark}

Theorem \ref{teo:laplace} may be compared to analogous statements in \cite{serfaty2023gaussian,peilen2024local,garcia2022generalized,berman2019sharp}. In the Coulomb or $1d$ log cases, the bound on the Laplace transform in \cite[Theorem 1.3]{peilen2024local} and \cite[Theorem 1]{serfaty2023gaussian} is a polynomial of higher order in $r$ \footnote{In the notation of \cite{serfaty2023gaussian,peilen2024local}, $r$ is replaced by $s$ or $t$, and the normalization is different} ($3$ and $4$ respectively), but with no constant term. Consequently, the estimate in Theorem \ref{teo:laplace} is better for $r$ tending to $\infty$, but worse as $r$ tends to $0$. Depending on $\kappa$ and $\lambda$, Theorem \ref{teo:laplace} may also require less regularity on the test function. 

We expect that the concentration bounds and estimates on the Laplace transform of fluctuations around the thermal equilibrium measure are more precise than the ones around the equilibrium measure in all temperature regimes since they don't contain the additional ${\rm ent}[\mu^{V}_{\infty}] - \log \left|\Sigma\right| + o_{N}(1)$ term. This has been proved for Riesz interactions and Coulomb interactions \cite{padilla2023concentration,garcia2022generalized}. Indeed, in the high-temperature regime, they are significantly more accurate even for general interactions. Proving this outside the high-temperature regime would require proving uniform bounds on the $W^{\alpha}$ norm of $\mu_{\theta}^{V}$ in terms of the $W^{\alpha}$ norm of $\mu_{\infty}^{V}$ uniformly as $\theta \to \infty$. In the Coulomb and Riesz case, an analogous statement is known to hold for the $L^{\infty}$ norm \cite{armstrong2022thermal, garcia2022generalized}.

\begin{remark}
    In this paper, we have deduced bounds on the Laplace transform of fluctuations from concentration bounds. It is also possible to deduce concentration bounds from bounds on the Laplace transform of fluctuations by applying Chebyshev's inequality. 
\end{remark}

\section{Further work}
\label{sect:further}

\begin{itemize}

\item[1.] A problem that has been widely studied in this field \cite{serfaty2023gaussian,leble2018fluctuations,bekerman2018clt,boursier2021optimal,peilen2024local,hardy2021clt,bauerschmidt2019two,lambert2019quantitative} is to obtain a Central Limit Theorem (CLT) for fluctuations, either ${\rm Fluct}_{\infty}[f] $ or ${\rm Fluct}_{\theta}[f] $. A remarkable fact is that the fluctuations are of magnitude $\frac{1}{N}$, i.e. $N {\rm Fluct}_{\infty}[f]$ converges to a nontrivial Gaussian random variable in some temperature regimes and dimensions. Furthermore, it is often possible to identify the mean and variance of the limiting Gaussian random variable. Since convergence to a Gaussian random variable is implied by the convergence of the log-Laplace transform to a quadratic function, this problem is closely linked to obtaining precise asymptotes of the Laplace transform of fluctuations. The asymptotes we obtain in this paper are not precise enough to obtain a CLT. However, the problem still makes sense in the context of general interactions, dimensions, and temperature regimes: are fluctuations of order $N$? What are the means and variance of the limiting Gaussian, as a function of $f$? 

\item[2.] The LDPs proved in \cite{garcia2019large,chafai2014first} imply that, a.s. under the Gibbs measure, the empirical measure converges to the equilibrium measure. This is equivalent to the statement that a.s. under the Gibbs measure, the bounded Lipschitz norm of the difference between the empirical measure and the equilibrium measure tends to $0$. A more precise and quantitative version of this statement in the case of Coulomb and Riesz interactions can be found in \cite{chafai2018concentration, padilla2023concentration, garcia2019concentration, garcia2022generalized}, which obtained upper bounds for the probability that the bounded Lipschitz norm of the difference is larger than a certain threshold. Obtaining such bounds in the context of general interactions remains an open problem. 

\item[3.]  Theorems \ref{teo:concup}, \ref{teo:conclow}, and \ref{teo:laplace} identify the leading order term in a certain open interval in $r$ if $f$ satisfies that $\mathrm{supp} h_{-1}(f) \subset \Sigma$. If $g$ is the Coulomb kernel, this is equivalent to the condition $\Delta f \subset \Sigma$, which itself is equivalent to the condition that $\mathrm{supp}(f) \subset \Sigma$. This is a common assumption (see \cite{bauerschmidt2019two, serfaty2023gaussian}). For arbitrary $r$ and $f$, the leading-order term can be obtained using the results in \cite{chafai2014first, garcia2019large}, but obtaining precise error estimates remains an open problem.    

\item[4.] In the Coulomb and Riesz cases, the thermal equilibrium measure provides a better approximation to the system (the empirical measure) than the equilibrium measure (see \cite{padilla2023concentration, garcia2022generalized}). In the Coulomb case, it can also be used to obtain better local laws \cite{armstrong2021local} and CLTs under more general hypotheses \cite{serfaty2023gaussian}. Such results rely on an understanding of the behavior of the thermal equilibrium measure $\mu_{\theta}$ for large $\theta$, and also on fine features of the convergence to $\mu_{\infty}$, see \cite{armstrong2022thermal}, and analogous results for Riesz interactions in \cite{garcia2022generalized}. Such results are currently not available for general interactions. These estimates would allow us to obtain new results for interacting particle systems and would be of interest in their own right.     

\end{itemize}

\section{Literature comparison}

This paper was originally inspired by the series of articles \cite{maida2014free,chafai2018concentration,garcia2019concentration,padilla2023concentration,garcia2022generalized}. \cite{chafai2018concentration}, building upon the work of \cite{maida2014free}, proves concentration of measure around the equilibrium measure for Coulomb gases in all dimensions, in the Wasserstein and bounded-Lipshitz metrics. In particular, this proves concentration of measure for Ginibre ensembles (as well other random matrix models). It crucially relies on new inequalities
between metrics on probability measures, including inequalities comparing the Coulomb energy with metrics coming from Optial Transportation. Building upon this work, \cite{garcia2019concentration} proves concentration of measure around the equilibrium measure in the Wasserstein norm for Coulomb gases in compact Riemannian manifolds. Also building upon \cite{chafai2018concentration}, \cite{padilla2023concentration} proves concentration of measure around the \emph{thermal} equilibrium measure in the Wasserstein norm, and shows that the concentration inequality is optimal in a certain geometric sense. 
\cite{garcia2022generalized} extends these ideas to Riesz, and a class of Riesz-type interactions introduced in \cite{nguyen2022mean, rosenzweig2023global}. In this setting, the ``transport" inequalities do not compare the electric energy to the Wasserstein or bounded-Lipschitz norms, but with dual-Holder norms. As in \cite{chafai2018concentration, garcia2019concentration, padilla2023concentration}, concentration of measure is proved around the equilibrium and thermal equilibrium measures. These inequalities are proved in the dual-Holder norms, but imply concentration of measure also in the Wasserstein norm. However, unlike the Coulomb case, these rates are not geometrically optimal. Leveraging on these results, \cite{garcia2022generalized} concludes by proving Moser-Trudinger-type inequalities, which may be interpreted as bounds on the Laplace transform of fluctuations. This paper expands the program in \cite{garcia2022generalized} to general interactions, with the exception of transport inequalities, which are not available for general interactions. The concentration inequalities proved here, unlike the ones in \cite{chafai2018concentration, garcia2019concentration, padilla2023concentration, garcia2022generalized} are not stated in terms of the norm of a difference (either between the empirical measure, and the equilibrium, or thermal equilibrium measures), but stated for a fixed test function. This approach is better suited to general interactions. Furthermore, we improve on the Moser-Trudinger-type inequalities by showing that they are saturated in some cases. \\

This paper may be compared to \cite{chafai2014first,garcia2019large} since they both deal with macroscopic observables governed by the Gibbs measure of a general interacting particle system. \cite{chafai2014first,garcia2019large} Derive an LDP for the empirical measure. The natural and general approach in \cite{garcia2019large} allows for extensions to many-body interactions and Riemannian manifolds, among other extensions. Obtaining an LDP is a very closely related question, and the results in \cite{garcia2019large} are the starting point in some of our proofs. However, this approach yields only the leading order terms in an asymptotic expansion. It is not possible to obtain a rate of convergence by means of it. For this, we must introduce a new approach.\\  

Another paper that deals with the Gibbs measure of a general interacting particle system, with the number of particles tending to infinity is \cite{lambert2021poisson}. The main result is that at high temperature $\left( \beta = \frac{1}{N} \right)$, the empirical field, defined as $\sum_{i=1}^{N} \delta_{N^{\frac{1}{d}} x_{i}}$, converges to a Poisson point process. The approach is to analyze the correlation function of the empirical field. \\

As mentioned in Section \ref{sect:further}, an important problem in this field is to obtain a CLT for fluctuations, and this reduces to finding precise asymptotes of the Laplace transform, see \cite{serfaty2023gaussian,leble2018fluctuations,bekerman2018clt,boursier2021optimal,peilen2024local,hardy2021clt,bauerschmidt2019two,johansson1998fluctuations,lambert2019quantitative}. The techniques used for this end may vary considerably: \cite{serfaty2023gaussian,leble2018fluctuations,bekerman2018clt,peilen2024local} uses the transport approach introduced in \cite{johansson1998fluctuations}. The approach in \cite{lambert2019quantitative,hardy2021clt,boursier2021optimal} is inspired by Stein's method. The approach in \cite{bauerschmidt2019two} is to analyze a loop equation for the Coulomb gas. \\     

Apart from being applied in CLTs, estimates for the partition function are an interesting problem in their own right. In the Riesz and Coulomb setting, it is possible to obtain very precise estimates by means of the \emph{renormalized energy} and \emph{screening procedure}. This approach was used in \cite{leble2017large, armstrong2021local,sandier20152d,rougerie2016higher,petrache2017next}.\\

\section{Preliminaries}
\label{sec:prelims}

In this section, we present some foundational results and definitions. These will be necessary to prove the main statements of the paper.

\begin{definition}
We begin by defining the bilinear form which induces the norm $\mathcal{E}$. The bilinear form $\mathcal{G}$ is defined for measures $\mu, \nu$ as 
\begin{equation}
    \mathcal{G} (\mu, \nu) = \int_{\mathbf{T}^{d} \times \mathbf{T}^{d}} g(x-y) \mathrm d  \mu \otimes \nu(x,y).
\end{equation}

We now present a convenient way of writing the energy of empirical measures. We define $\mathcal{E}^{\neq}$ for a measure $\mu$ as 
\begin{equation}
    \mathcal{E}^{\neq} (\mu) = \iint_{\mathbf{T}^{d} \times \mathbf{T}^{d} \setminus \Delta} g(x-y) \mathrm d \mu\otimes  \mu (x,y),
\end{equation}
where 
\begin{equation}
    \Delta = \{ (x,x) \in \mathbf{T}^{d} \times \mathbf{T}^{d} \}.
\end{equation}
We also introduce notation for the energy of an empirical measure with a background. Given a measure $\mu$, and $X_{N} \in \mathbf{T}^{d \times N}$, we define ${\rm F}_{N}(X_{N}, \mu)$ as 
\begin{equation}
    {\rm F}_{N}(X_{N}, \mu) = \mathcal{E}^{\neq} (\mu - {\rm emp}_{N}).
\end{equation}
The reason for excluding the diagonal in the integral is to avoid the singularity at the origin in $g$. This way, these quantities can be finite on empirical measures, and thus accurately reflect their electric energy. 

\end{definition}

Next, we move on to the existence and uniqueness of the (thermal) equilibrium measure, and to a characterization. 

\begin{lemma}
Assume $V$ and $g$ are admissible. Then the functional $\mathcal{E}_{V}$ has a unique minimizer in the set of probability measures, which has compact support $\Sigma$ and satisfies the First Order Condition
\begin{equation}\label{ELeq}
    \begin{split}
        2h({\mu^{V}_{\infty}}) + V - c_{\infty} &\geq 0 \text{ in } \mathbf{T}^{d}\\
        2h({\mu^{V}_{\infty}}) + V - c_{\infty} &= 0 \text{ in } \Sigma, 
    \end{split}
\end{equation}
for some constant $c_{\infty} \in \mathbf{R}.$
\end{lemma}

\begin{proof}
    See \cite[Chapter 2]{serfaty2015coulomb} for the Coulomb case, which extends to the general case. 
\end{proof}

\begin{lemma}
The functional $\mathcal{E}_{V}^{\theta}$ has a unique minimizer in the set of probability measures, which is everywhere positive, and  satisfies the First Order Condition
\begin{equation}
\label{eq:ELeqtherm}
    2h({\mu^{V}_{\theta}}) + V + \frac{1}{\theta} \log (\mu^{V}_{\theta}) = c_{\theta},
\end{equation}
for some constant $c_\theta$.
\end{lemma}

\begin{proof}
    See \cite[Proposition 5.4]{padilla2022large}.
\end{proof}

We move on to splitting formulae, which are a way of expanding the (thermal) energy around its value at the (thermal) equilibrium measure. 

\begin{proposition}[Splitting formula]
\label{Prop:splitting}
For every $X_N = (x_1,\dots,x_N) \in \mathbf T^{d \times N}$ 
such that $x_i \neq x_j$ whenever $i \neq j$, we have that the Hamiltonian $\mathcal{H}_{N}$ can be split into 
\begin{equation}
\label{eq:splittingform}
    \mathcal{H}_{N}(X_{N}) = N^{2} \left( \mathcal{E}_{V}(\mu^{V}_{\infty}) + {\rm F}_{N}(X_{N}, \mu^{V}_{\infty}) + \int_{\mathbf{T}^{d}} \zeta_{\infty}\, \mathrm d\, {\rm emp}_{N} \right),
\end{equation}
where 
\begin{equation}
    \zeta_{\infty} = V +c_{\infty}- 2 h({\mu^{V}_{\infty}}) .
\end{equation}
\end{proposition}

Note that $\zeta_{\infty}$ depends on the potential $V$, but this dependence is omitted for ease of notation. 

\begin{proof}
See \cite[Section 3]{sandier20152d} for the Coulomb case, which extends to the general case.
\end{proof}

\begin{proposition}[Thermal splitting formula]
We introduce the notation
\begin{equation}
    \zeta_{\theta} = - \frac{1}{\theta} \log\left(\mu^{V}_{\theta}\right).
\end{equation}
For every $X_N = (x_1,\dots,x_N) \in \mathbf T^{d \times N}$ 
such that $x_i \neq x_j$ whenever $i \neq j$ and every $\theta > 0$, we have that the Hamiltonian $\mathcal{H}_{N}$ can be split into  
\begin{equation}\label{eq:thermspltfrm}
    \mathcal{H}_{N} (X_{N}) = N^{2} \left( \mathcal{E}_{V}^{\theta} (\mu^{V}_{\theta}) +  {\rm F}_{N}(X_{N}, \mu^{V}_{\theta})+ \int_{\mathbf{T}^{d}} \zeta_{\theta}\, \mathrm d\, {\rm emp}_{N} \right).
\end{equation}
\end{proposition}

\begin{proof}
    See \cite[Propostion 5.5]{padilla2022large}.
\end{proof}

As an application of the splitting formulae, we may factor out the leading-order components of the partition function, and use this to more conveniently rewrite the Gibbs measure. 

\begin{corollary}
\label{cor:Gibbs}
The Gibbs measure may be rewritten as 
\begin{equation}
   \mathrm d \mathbf{P}_{N, \beta} (X_{N}) = \frac{1}{K_{N, \beta}^{\infty}} \exp \left( - N^{2} \beta \left[ {\rm F}_{N}(X_{N}, \mu^{V}_{\infty}) \right]  \right) \Pi_{i=1}^{N} \rho_{N}(x_{i}) \, \mathrm d X_{N},
\end{equation}
where
\begin{equation}
\label{eq:rho}
    \rho_{N}(x) := \exp \left( - \theta  \zeta_{\infty}(x) \right), 
\end{equation}
$\theta := N \beta$, and 
\begin{equation}
    K_{N, \beta}^{\infty} = \frac{ Z^{V}_{N, \beta} }{\exp \left( N^{2} \beta \mathcal{E}_{V}(\mu^{V}_{\infty}) \right)}.
\end{equation}

Alternatively, the Gibbs measure may be rewritten as 
\begin{equation}
   \mathrm d \mathbf{P}_{N, \beta} (X_{N}) = \frac{1}{K_{N, \beta}^{\theta}} \exp \left( - N^{2} \beta \left[ {\rm F}_{N}(X_{N}, \mu^{V}_{\theta}) \right]  \right) \Pi_{i=1}^{N} \mu_{\theta}^{V}(x_{i}) \, \mathrm d X_{N},
\end{equation}
where
\begin{equation}
    K_{N, \beta}^{\theta} = \frac{ Z^{V}_{N, \beta} }{\exp \left( N^{2} \beta \mathcal{E}_{V}^{\theta}(\mu^{V}_{\theta}) \right)}.
\end{equation}    
\end{corollary}

\section{Proof of asymptotes of the partition functions}
\label{sec:asymps}

In this section, we give the proof of Lemma \ref{lem:partfunc}. We will use the variational characterization of the Gibbs measure, see for example \cite[Chapter 15]{georgii2011gibbs}:
\begin{equation}
\label{eq:varcar}
    -\frac{1}{\beta} \log Z^{V}_{N, \beta} = \min_{\mu \in \mathcal{P}(\mathbf{T}^{d \times N})} \mathcal{F}_{N, \beta} (\mu),
\end{equation}
where
\begin{equation}
    \mathcal{F}_{N, \beta} (\mu) = \int_{\mathbf{T}^{d \times N}} \mathcal{H}_{N}(X_{N}) \mu(X_{N}) \, {\rm d} X_{N} + \frac{1}{\beta} \int_{\mathbf{T}^{d \times N}} \mu(X_{N}) \log (\mu(X_{N})) \, {\rm d} X_{N},
\end{equation}
and $\mathcal{P}(\mathbf{T}^{d \times N})$ denotes the set of probability measures on $\mathbf{T}^{d \times N}$. 

We will also use a new regularization procedure, which is a new technical ingredient. 
\begin{lemma}
\label{lem:regularization}
Let $N>1$, $X_{N} \in \mathbf{T}^{d \times N}$, $g: \mathbf{T}^{d}$ be admissible, and $\alpha>0$. Then for all $t>0$ there exists ${\rm emp}_{N}^{t} \in \mathcal{P}(\mathbf{T}^{d})$ such that
\begin{equation}
        \mathcal{E}\left( {\rm emp}_{N}^{t} \right) \leq  \mathcal{E}^{\neq}\left( {\rm emp}_{N} \right) + C \left( \frac{t^{1-\frac{d}{\gamma}} }{N} +t \right),
\end{equation}
\begin{equation}
       \left| \int_{\mathbf{T}^{d}} f \, \mathrm{d} \left( {\rm emp}_{N}^{t} - {\rm emp}_{N} \right) \right| \leq C \max \left\{ t, t^{\frac{\alpha}{\lambda}} \right\} \|f\|_{W^{\alpha}},
\end{equation}
for some $C>0$ depending only on $g$. We call ${\rm emp}_{N}^{t} $ the {regularized empirical measure}. 
\end{lemma}
The proof is found in Section \ref{sect:regularization}.

We will now begin the proof of Lemma \ref{lem:partfunc}, restated here:
\begin{lemma}
Assume that $V$ and $g$ are admissible and let $\alpha \geq \gamma$. Define $\theta = N \beta$. Then
\begin{itemize}
\item[a)] \begin{equation}
    -\mathcal{E}_{V}^{\theta}(\mu^{V}_{\theta}) + \frac{1}{N} \mathcal{E}(\mu^{V}_{\theta}) \leq \frac{\log Z^{V}_{N, \beta}}{N^{2} \beta} \leq -\mathcal{E}_{V}^{\theta}(\mu^{V}_{\theta}) +C\left( 1 +\|\mu^{V}_{\theta}\|_{W^{\alpha - \gamma}} \right)N^{p^{*}},
\end{equation}
where 
\begin{equation}
    p^{*} = p^{*} (\alpha, \lambda, \gamma) := \max \left\{ -\frac{\gamma}{d}, \left( \frac{\lambda}{\alpha} - \frac{\lambda d}{\alpha \gamma} -1 \right)^{-1} \right\}.
\end{equation}

    \item[b)] If $\lim_{N \to \infty} \theta = \infty$, then
    \begin{equation}
    \begin{split}
    &-\mathcal{E}_{V}(\mu^{V}_{\infty}) - \frac{1}{\theta} {\rm ent}[\mu^{V}_{\infty}] + \frac{1}{N} \mathcal{E}(\mu^{V}_{\infty}) \leq \frac{\log Z^{V}_{N, \beta}}{N^{2} \beta} \leq\\
    &-\mathcal{E}_{V}(\mu^{V}_{\infty})+C\left( 1+\|\mu^{V}_{\infty}\|_{W^{\alpha - \gamma}} \right)N^{p^{*}}  - \frac{1}{\theta} \left( \log |\Sigma| + o_{N}(1)  \right).
    \end{split}
\end{equation}
\end{itemize}

\end{lemma}

\begin{proof}

\textbf{Step 1:} Proof of lower bounds.

Taking $(\mu^{V}_{\infty})^{\otimes N}$ as a test function in equation \eqref{eq:varcar}, we obtain
\begin{equation}
\label{eq:test}
    -\frac{1}{\beta} \log Z^{V}_{N, \beta} \leq \int_{\mathbf{T}^{d \times N}} \mathcal{H}_{N}(X_{N}) \Pi_{i=1}^{N} \mu^{V}_{\infty}(x_{i}) \, {\rm d} x_{i} + \frac{N}{\beta} \int_{\mathbf{T}^{d \times N}} \mu^{V}_{\infty}(x) \log (\mu^{V}_{\infty}(x)) \, {\rm d} x.
\end{equation}

Expanding $\mathcal{H}_{N}(X_{N})$ and carrying out the integrals, equation \eqref{eq:test} can be simplified to
\begin{equation}
    -\frac{1}{\beta} \log Z^{V}_{N, \beta} \leq N^{2} \mathcal{E}_{V}(\mu^{V}_{\infty}) - \mathcal{E}(\mu^{V}_{\infty})  + \frac{N}{\beta} {\rm ent}[\mu_{\infty}^{V}],    
\end{equation}
which yields
\begin{equation}
   \log Z^{V}_{N, \beta} \geq  -N^{2} \beta \mathcal{E}_{V}(\mu^{V}_{\infty}) + \beta \mathcal{E}(\mu^{V}_{\infty}) - N {\rm ent}[\mu_{\infty}^{V}].
\end{equation}

Similarly, taking $(\mu^{V}_{\theta})^{\otimes N}$ as a test function in equation \eqref{eq:varcar}, we obtain 
\begin{equation}
     -\frac{1}{\beta} \log Z^{V}_{N, \beta} \leq N^{2} \mathcal{E}_{V}^{\theta}(\mu^{V}_{\theta}) - \mathcal{E}(\mu^{V}_{\theta}).
\end{equation}

\textbf{Step 2:} Proof of upper bound.

For the remainder of the proof, we will introduce the notation
\begin{equation}
    z^{V}_{N, \beta} := \int_{\mathbf{T}^{d}} \exp \left( - \theta \zeta_{\infty} (x) \right) \, {\rm d} x, 
\end{equation}
and
\begin{equation}
    p_{N, \beta} = \frac{1}{z^{V}_{N, \beta}} \rho_{N},
\end{equation}
with $\rho_{N}$ given by equation \eqref{eq:rho}. 

Let $t>0$ to be determined later, and let ${\rm emp}_{N}^{t}$ denote the regularized empirical measure (see Lemma \ref{lem:regularization}). By Lemma \ref{lem:regularization}, we have that
\begin{equation}
\label{eq:bound}
\begin{split}
    {\rm F}_{N}(X_{N}, \mu^{V}_{\infty}) & = \mathcal{E}^{\neq}({\rm emp}_{N}) -2 \int_{\mathbf{T}^{d}} h(\mu^{V}_{\infty}) \, \mathrm d {\rm emp}_{N} + \mathcal{E}(\mu^{V}_{\infty})\\
    &\geq \mathcal{E}({\rm emp}_{N}^{t} - \mu^{V}_{\infty}) - C \left( \frac{t^{1-\frac{d}{\gamma}} }{N} +t \right)-C \max \left\{ t, t^{\frac{\alpha}{\lambda}} \right\} \|h(\mu^{V}_{\infty})\|_{W^{\alpha}}.
\end{split}    
\end{equation}
Note that, by assumption 2 of Definition \ref{def:admissible}, 
\begin{equation}
    \|h(\mu^{V}_{\infty})\|_{W^{\alpha}} \leq C \|\mu^{V}_{\infty}\|_{W^{\alpha - \gamma}}.
\end{equation}

The optimal $t$ in equation \eqref{eq:bound} is given by $t=N^{-\frac{\gamma}{d}}$ if $\alpha \geq \lambda$; and $t = N^{\left(1 - \frac{d}{\gamma} - \frac{\alpha}{\lambda}\right)^{-1}}$ if $\alpha < \lambda$. The result is that 
\begin{equation}
\begin{split}
   {\rm F}_{N}(X_{N}, \mu^{V}_{\infty}) &\geq \mathcal{E}({\rm emp}_{N}^{t} - \mu^{V}_{\infty}) - C \left(  1+ \|\mu^{V}_{\infty}\|_{W^{\alpha - \gamma}} \right) \max\left\{N^{-\frac{\gamma}{d}},  N^{\left( \frac{\lambda}{\alpha} - \frac{\lambda d}{\alpha \gamma} -1 \right)^{-1}} \right\}\\
   &\geq- C\left(  1+ \|\mu^{V}_{\infty}\|_{W^{\alpha - \gamma}} \right)N^{p^{*}}.
\end{split}   
\end{equation}

We then have that, for any probability measure $\mu \in \mathcal{P}(\mathbf{T}^{d \times N})$,
\begin{equation}
    \begin{split}
        &\mathcal{F}_{N, \beta}(\mu) \\
         \geq &N^{2}\left(\mathcal{E}_{V}(\mu^{V}_{\infty})-C\left(  1+ \|\mu^{V}_{\infty}\|_{W^{\alpha - \gamma}} \right)N^{p^{*}}\right) + N \int_{\mathbf{T}^{d \times N}} \sum_{i=1}^{N} \zeta_{\infty}(x_{i}) \, {\rm d} X_{N} \\
         & \ \ \ \ \ \ + \frac{1}{\beta} \int_{\mathbf{T}^{d \times N}} \mu(X_{N}) \log (\mu(X_{N})) \, {\rm d} X_{N}\\
         = & N^{2}\left(\mathcal{E}_{V}(\mu^{V}_{\infty})-C\left(  1+ \|\mu^{V}_{\infty}\|_{W^{\alpha - \gamma}} \right)N^{p^{*}}\right) + \frac{1}{\beta} \int_{\mathbf{T}^{d \times N}}  \sum_{i=1}^{N} \log \left( z_{N, \beta}\frac{\exp \left( \theta \zeta_{\infty}(x_{i}) \right)}{z_{N, \beta}} \right) \, {\rm d} X_{N} \\
         & \ \ \ \ \ \ + \frac{1}{\beta} \int_{\mathbf{T}^{d \times N}} \mu(X_{N}) \log (\mu(X_{N})) \, {\rm d} X_{N}\\
        = & N^{2}\left(\mathcal{E}_{V}(\mu^{V}_{\infty})-C\left(  1+ \|\mu^{V}_{\infty}\|_{W^{\alpha - \gamma}} \right)N^{p^{*}}\right) + \frac{N}{\beta} \log (z_{N, \beta}) + \frac{1}{\beta} {\rm ent}[\mu | p_{N, \beta}^{\otimes N}].
    \end{split}
\end{equation}
where the relative entropy of two probability measures, $\mu, \nu$ is defined as 
\begin{equation}
    {\rm ent}[\mu | \nu] := \int \mu \log \left( \frac{\mu}{\nu} \right) \, \mathrm d x.
\end{equation}

A simple, well-known application of Jensen's inequality implies that
\begin{equation}
    {\rm ent}[\mu | p_{N, \beta}^{\otimes N}] \geq 0.
\end{equation}
On the other hand, if $\lim_{N \to \infty} \theta = \infty$, then
\begin{equation}
    \lim_{N \to \infty} \log (z_{N, \beta}) = \log |\Sigma|.
\end{equation}

This yields
\begin{equation}
    \log Z^{V}_{N, \beta} \leq N^{2} \beta \left( -\mathcal{E}_{V}(\mu^{V}_{\infty}) + C \left( 1 + \|\mu^{V}_{\infty}\|_{W^{\alpha - \gamma}} \right) N^{p^{*}} \right) - N \left( \log |\Sigma| + o(1)  \right).
\end{equation}

The proof of the upper bound in terms of $\mathcal{E}_{V}^{\theta }(\mu_{\theta}^{V})$ is similar: we use the regularization procedure of Lemma \ref{lem:regularization} to obtain a lower bound for $\mathrm{F}_{N} \left( X_{N}, \mu_{\theta}^{V} \right)$, and then this lower bound to obtain a lower bound for $\mathcal{F}_{N, \beta}(\mu) $ for arbitrary $\mu$. 
\end{proof}

\section{Proof of concentration bounds}
\label{sec:conccentration}

In this section, we will show the proof of Theorems \ref{teo:concup} and \ref{teo:conclow}. The proof will require the following lemma:
\begin{lemma}
\label{lem:firstconc}
Let $r \in \mathbf{R}^{+}$ and $X_{N} \in \mathbf{T}^{d \times N}$ be a pair-wise distinct  point configuration. Then
\begin{equation}
\label{eq:thermal}
   \mathbf{P}_{N, \beta} \left( {\rm F}_{N}(X_{N}, \mu^{V}_{\theta}) \geq r \right) \leq \exp \left( - N^{2} \beta r \right)
\end{equation}
and, if $\lim_{N \to \infty} \theta = \infty$,
\begin{equation}
\label{eq:equilibrium}
   \mathbf{P}_{N, \beta} \left( {\rm F}_{N}(X_{N}, \mu^{V}_{\infty}) \geq r \right) \leq \exp \left( - N^{2} \beta r + N \left( {\rm ent}[\mu^{V}_{\infty}] - \log|\Sigma| + o_{N}(1) \right) \right).
\end{equation}
\end{lemma}

\begin{proof}
We first prove equation \eqref{eq:equilibrium}. Using Corollary \ref{cor:Gibbs} and Lemma \ref{lem:partfunc} we have that
\begin{equation}
    \begin{split}
        &\mathbf{P}_{N, \beta} \left( {\rm F}_{N}(X_{N}, \mu^{V}_{\infty}) \geq r \right) \\
        =  &\frac{1}{K_{N, \beta}^{\infty}} \int_{{\rm F}_{N}(X_{N}, \mu^{V}_{\infty}) \geq r} \exp \left( - N^{2} \beta \left[ {\rm F}_{N}(X_{N}, \mu^{V}_{\infty}) \right]  \right) \Pi_{i=1}^{N} \rho_{N}(x_{i}) \, \mathrm d X_{N} \\
        \leq & \exp \left( N {\rm ent}[\mu^{V}_{\infty}] \right) \int_{{\rm F}_{N}(X_{N}, \mu^{V}_{\infty}) \geq r} \exp \left( - N^{2} \beta \left[ {\rm F}_{N}(X_{N}, \mu^{V}_{\infty}) \right]  \right) \Pi_{i=1}^{N} \rho_{N}(x_{i}) \, \mathrm d X_{N} \\
        \leq & \exp \left( N ({\rm ent}[\mu^{V}_{\infty}] - \log z_{N} )  \right) \exp \left( - N^{2} \beta r  \right) \int_{{\rm F}_{N}(X_{N}, \mu^{V}_{\infty}) \geq r} \Pi_{i=1}^{N} p_{N, \beta}(x_{i}) \, \mathrm d X_{N} \\
        \leq & \exp \left( - N^{2} \beta r+N ({\rm ent}[\mu^{V}_{\infty}] - \log z_{N} ) \right). 
    \end{split}
\end{equation}

Since $\lim_{N \to \infty} z_{N} = |\Sigma|$, we can conclude. 

The proof of equation \eqref{eq:thermal} is analogous. 
\end{proof}

We will also use another lemma, whose proof is postponed to the appendix. 
\begin{lemma}
\label{lem:construction}
Let $\varphi$ be a non-negative function such that $\int_{\mathbf{T}^{d}} \varphi =1$ and $\supp{\varphi} \subset \Sigma$, $f$ be a continuous function, and $g:\mathbf{T}^{d}\to \mathbf{R}$ be an admissible kernel. Then for any $p < \frac{1}{d}$ and $\tau >0$ there exists a family of configurations $\Lambda_{N} \subset \mathbf{T}^{d \times N}$ such that
\begin{itemize}
    \item[1.] For any $X_{N} \in \Lambda_{N}$,
    \begin{equation}
        \left| {\rm F}_{N}(X_{N}, \mu^{V}_{\infty}) - \mathcal{E}(\varphi - \mu^{V}_{\infty}) \right| \leq N^{p(s-d)}.
    \end{equation}
    
    \item[2.] \begin{equation}
        \log \left( p_{N, \beta}^{\otimes N} \left( \Lambda_{N} \right) \right) \geq -N \left( {\rm ent}[\varphi| p_{N, \beta}] + \tau \right).
    \end{equation}
    
    \item[3.] For any $X_{N} \in \Lambda_{N}$ and $\alpha>0$,
    \begin{equation}
    \label{eq:testerror}
       \left| \frac{1}{N} \sum_{i=1}^{N} f(x_{i}) - \int_{\mathbf{T}^{d}} f \varphi \, {\rm d} x \right| \leq 
        C N^{- \alpha p }  \|f\|_{W^{\alpha}}.
    \end{equation}
\end{itemize}
\end{lemma}

The proof is found in Section \ref{sec:app2}.

We will also use the regularization procedure, see Lemma \ref{lem:regularization}. We now continue to the proof of Theorems \ref{teo:concup} and \ref{teo:conclow}. 
\begin{proof} 

\textbf{Step 1:} Proof of bounds for weakly admissible kernels.

We start by giving the proof of equation \eqref{eq:upboundweak}. As a consequence of the LDP proved in \cite{garcia2019large}, we have that
\begin{equation}
\begin{split}
    \limsup_{N \to \infty} \frac{1}{N^{2} \beta} \log \left( \mathbf{P}_{N, \beta} \left( \left| {\rm Fluct}_{\infty}[f]  \right| \geq r \right) \right) &= - \inf_{ \mu \in \mathcal{P}(\mathbf{T}^{d}), \left| \int_{\mathbf{T}^{d}} f \, \mathrm d ( \mu - \mu_{\infty}^{V})  \right| \geq r} \mathcal{E}_{V}(\mu) - \inf_{ \mu \in \mathcal{P}(\mathbf{T}^{d})} \mathcal{E}_{V}(\mu) \\
    &= - \inf_{ \mu \in \mathcal{P}(\mathbf{T}^{d}), \left| \int_{\mathbf{T}^{d}} f \, \mathrm d ( \mu - \mu_{\infty}^{V})  \right| \geq r} \mathcal{E}_{V}(\mu) - \mathcal{E}_{V}(\mu_{\infty}^{V}).  
\end{split}    
\end{equation}

We will use the notation $S := \{ \mu \in \mathcal{P}(\mathbf{T}^{d}), \left| \int_{\mathbf{T}^{d}} f \, \mathrm d ( \mu - \mu_{\infty}^{V}) \right| \geq r\}$,  $S_{+} := \{ \mu \in \mathcal{P}(\mathbf{T}^{d}),  \int_{\mathbf{T}^{d}} f \, \mathrm d ( \mu - \mu_{\infty}^{V}) \geq r\}$, and $S_{-} := \{ \mu \in \mathcal{P}(\mathbf{T}^{d}), \int_{\mathbf{T}^{d}} f \, \mathrm d ( \mu - \mu_{\infty}^{V})  \leq -r\}$. Note that $S = S_{+} \cup S_{-}$ and also that $S_{+}$ and $S_{-}$ are convex and closed in the topology of weak convergence. Therefore the minimum of $\mathcal{E}_{V}$ on $S$ is achieved by a probability measure, denoted $\mu_{*}$.  

We claim that $\mu_{*}$ lies on the boundary of $S$, i.e. it satisfies that $\left| \int_{\mathbf{T}^{d}} f \, \mathrm d ( \mu_{*} - \mu_{\infty}^{V}) \right| = r$. To show this claim, we proceed by contradiction and assume that $\mu_{*}$ lies in the interior of $S$. Then $\mu_{*}$ is a critical point of $\mathcal{E}_{V}$, which by uniqueness implies that $\mu_{*} = \mu_{\infty}^{V}$, but $\mu_{\infty}^{V} \notin S$. This is a contradiction and therefore $\left| \int_{\mathbf{T}^{d}} f \, \mathrm d ( \mu_{*} - \mu_{\infty}^{V}) \right| = r$. For the remainder of the proof, we will assume that $\int_{\mathbf{T}^{d}} f \, \mathrm d ( \mu_{*} - \mu_{\infty}^{V}) = r$. If we assume that $\int_{\mathbf{T}^{d}} f \, \mathrm d ( \mu_{*} - \mu_{\infty}^{V}) = -r$, the proof is analogous.

In order to find $\mu_{*}$, we introduce the Lagrange multiplier:
\begin{equation}
\label{eq:lagrangemult}
 \mathcal{E}_{V}(\mu) - \mathcal{E}_{V}(\mu_{\infty}^{V}) + \lambda \left(  \int_{\mathbf{T}^{d}} f \, \mathrm d ( \mu - \mu_{\infty}^{V}) -r \right).     
\end{equation}
A standard computation in the Calculus of Variations shows that equation \eqref{eq:lagrangemult} has the same critical values as 
\begin{equation}
 \min_{ \mu \in \mathcal{P}(\mathbf{T}^{d}),  \int_{\mathbf{T}^{d}} f \, \mathrm d ( \mu - \mu_{\infty}^{V}) = r } \mathcal{E}_{V}(\mu) - \mathcal{E}_{V}(\mu_{\infty}^{V}).     
\end{equation}

Using the splitting formula (Proposition \ref{Prop:splitting}) and the fact that $\zeta_{\infty}$ is non-negative, we have that 
\begin{equation}
\label{eq:upboundlagrange}
    \begin{split}
         &\min_{ \mu \in \mathcal{P}(\mathbf{T}^{d})} \mathcal{E}_{V}(\mu) - \mathcal{E}_{V}(\mu_{\infty}^{V}) + \lambda \left(  \int_{\mathbf{T}^{d}} f \, \mathrm d ( \mu - \mu_{\infty}^{V}) -r \right) \\
         =&  \min_{ \mu \in \mathcal{P}(\mathbf{T}^{d})} \mathcal{E}(\mu - \mu_{\infty}^{V}) + \int_{\mathbf{T}^{d}} \zeta_{\infty} \, \mathrm d \mu + \lambda \left(  \int_{\mathbf{T}^{d}} f \, \mathrm d ( \mu - \mu_{\infty}^{V}) -r \right) \\
         \geq &  \min_{ \mu \in \mathcal{P}(\mathbf{T}^{d})} \mathcal{E}(\mu - \mu_{\infty}^{V}) + \lambda \left(  \int_{\mathbf{T}^{d}} f \, \mathrm d ( \mu - \mu_{\infty}^{V}) -r \right) \\
          \geq &  \min_{ \mu \in \mathcal{M}(\mathbf{T}^{d})} \mathcal{E}(\mu - \mu_{\infty}^{V}) + \lambda \left(  \int_{\mathbf{T}^{d}} f \, \mathrm d ( \mu - \mu_{\infty}^{V}) -r \right), 
    \end{split}
\end{equation}
where $\mathcal{M}(\mathbf{T}^{d})$ denotes the space of measures (not necessarily probability measures) on $\mathbf{T}^{d}$. 

The minimizer of the last line of equation \eqref{eq:upboundlagrange}, denoted $\overline{\mu}$ can be computed explicitly: it is characterized by
\begin{equation}
\label{eq:explicitform}
    2 h (\overline{\mu} - \mu_{\infty}^{V}) = \lambda f, \ \ \ \  \int_{\mathbf{T}^{d}} f \, \mathrm d ( \mu - \mu_{\infty}^{V}) = r.
\end{equation}
We may rewrite the left equation in \eqref{eq:explicitform} as 
\begin{equation}
\label{eq:formmu*}
    \overline{\mu} = \mu_{\infty}^{V} + \frac{\lambda}{2} h^{-1}(f),
\end{equation}
while Plancherel's Theorem implies that
\begin{equation}
    \lambda = \frac{2r}{\|h_{-\frac{1}{2}}(f)\|_{L^{2}}^{2}}.
\end{equation}

Using the fact that 
\begin{equation}
    \mathcal{E}(h_{-1}(f)) = \int_{\mathbf{T}^{d}} f h_{-1}(f) = \|h_{-\frac{1}{2}}(f)\|_{L^{2}}^{2},
\end{equation}
(which can be once again verified by Plancherel's Theorem) we have that 
\begin{equation}
 \mathcal{E}_{V}(\overline{\mu}) - \mathcal{E}_{V}(\mu_{\infty}^{V}) = \frac{r}{\|h_{-\frac{1}{2}}(f)\|_{L^{2}}^{2}},    
\end{equation}
and equation \eqref{eq:upboundweak} follows. 

In order to prove equation \eqref{eq:loboundweak}, we note that, since $ \widehat{h_{-1}(\mu)} (0):= 0$, $\overline{\mu}$ given by \eqref{eq:formmu*} has integral $1$. Furthermore, if $f$ satisfies items $i)$ and $ii)$, then $\overline{\mu}$ is a probability measure supported in $\Sigma$. Therefore, proceeding as in equation \eqref{eq:upboundlagrange} and using the fact that $\zeta_{\infty}$ is $0$ in $\Sigma$, 
\begin{equation}
    \begin{split}
         \min_{ \mu \in \mathcal{P}(\mathbf{T}^{d}),  \left| \int_{\mathbf{T}^{d}} f \, \mathrm d ( \mu - \mu_{\infty}^{V}) \right| = r } \mathcal{E}_{V}(\mu) - \mathcal{E}_{V}(\mu_{\infty}^{V}) &=  \min_{ \mu \in \mathcal{M}(\mathbf{T}^{d}),  \left| \int_{\mathbf{T}^{d}} f \, \mathrm d ( \mu - \mu_{\infty}^{V}) \right| = r } \mathcal{E}_{V}(\mu) - \mathcal{E}_{V}(\mu_{\infty}^{V})\\
         &=  \mathcal{E}(\overline{\mu} - \mu_{\infty}^{V}).
    \end{split}
\end{equation}

\textbf{Step 2:} Proof of upper bounds for admissible kernels.

For the rest of the proof, we will use the following notation: given $t>0$, we denote
\begin{equation}
    {\rm Fluct}^{t}_{\infty}[f]  := \int_{\mathbf{T}^{d}} f d \left( {\rm emp}^{t}_{N} - \mu^{V}_{\infty} \right), 
\end{equation}
where $ {\rm emp}^{t}_{N}$ denotes the regularized empirical measure (see Lemma \ref{lem:regularization}).

We first prove equation \eqref{eq:conceq}. We start by noting that
\begin{equation}
\label{eq:testerror1}
    \begin{split}
         \left| {\rm Fluct}^{t}_{\infty}[f]  -  {\rm Fluct}_{\infty}[f] \right| &= \int_{\mathbf{T}^{d}} f d \left( {\rm emp}^{t}_{N} - {\rm emp}_{N} \right) \\
         &\leq  C \max \left\{ t, t^{\frac{\kappa}{\lambda}} \right\} \|f\|_{W^{\kappa}}.
    \end{split}
\end{equation}

Using Plancherel and Cauchy-Schwartz Theorems, we have that
\begin{equation}
    \begin{split}
        \int_{\mathbf{T}^{d}} f d \left( {\rm emp}^{t}_{N} - \mu^{V}_{\infty} \right) &= \sum_{m \in \mathbf{Z}^{d}} \widehat{f}(m) \left( \widehat{{\rm emp}^{t}_{N}}(m) - \widehat{\mu^{V}_{\infty}}(m) \right) \\ 
        &= \sum_{m \in \mathbf{Z}^{d} \setminus \{0\}} \widehat{g}^{-\frac{1}{2}}(m) \widehat{f}(m) \left( \widehat{{\rm emp}^{t}_{N}}(m) - \widehat{\mu^{V}_{\infty}}(m) \right)\widehat{g}^{\frac{1}{2}}(m) \\ 
        & \leq \left\| \widehat{g}^{-\frac{1}{2}} \widehat{f} \right\|_{L^{2}} \left\|  \left( \widehat{{\rm emp}^{t}_{N}} - \widehat{\mu^{V}_{\infty}}\right)\widehat{g}^{\frac{1}{2}} \right\|_{L^{2}} \\ 
         &= \left\| h_{-\frac{1}{2}}(f) \right\|_{L^{2}} \sqrt{\mathcal{E}\left({\rm emp}^{t}_{N} - \mu^{V}_{\infty}\right) }.
    \end{split}
\end{equation}
In order to pass from the first to the second line, we have used that, since ${{\rm emp}^{t}_{N}}$ and ${\mu^{V}_{\infty}}$ are probability measures, $\widehat{{\rm emp}^{t}_{N}}(0) - \widehat{\mu^{V}_{\infty}}(0) =0$. 

We then have that
\begin{equation}
\label{eq:reallylong}
    \begin{split}
        &\mathbf{P}_{N, \beta} \left( \left| {\rm Fluct}_{\infty}[f]  \right| \geq r \right)\\
        \leq &\mathbf{P}_{N, \beta} \left( \left| {\rm Fluct}^{t}_{\infty}[f]   \right| \geq r- C \|f\|_{W^{\kappa}}\max \left\{ t, t^{\frac{\kappa}{\lambda}} \right\}  \right)\\
        \leq &\mathbf{P}_{N, \beta} \left( \left\| h_{-\frac{1}{2}}(f) \right\|_{L^{2}} \sqrt{\mathcal{E}\left({\rm emp}^{t}_{N} - \mu^{V}_{\infty}\right) } \geq r - C \|f\|_{W^{\kappa}} \max \left\{ t, t^{\frac{\kappa}{\lambda}} \right\}  \right)\\
        = &\mathbf{P}_{N, \beta} \left( \sqrt{\mathcal{E}\left({\rm emp}^{t}_{N} - \mu^{V}_{\infty}\right) } \geq \frac{r- C \|f\|_{W^{\kappa}} \max \left\{ t, t^{\frac{\kappa}{\lambda}} \right\} }{\left\| h_{-\frac{1}{2}}(f) \right\|_{L^{2}}} \right)\\
        \leq &\mathbf{P}_{N, \beta} \Bigg( {\rm F}_{N}(X_{N}, \mu^{V}_{\infty}) \geq \frac{ \left(r- C \|f\|_{W^{\kappa}}\max \left\{ t, t^{\frac{\kappa}{\lambda}} \right\}  \right)^{2} }{\left\| h_{-\frac{1}{2}}(f) \right\|^{2}_{L^{2}}} - \\
        &\ \ \ \ \ \ \left(  C \left( \frac{t^{1-\frac{d}{\gamma}} }{N} +t \right)+C  \left( 1+  \|\mu^{V}_{\infty}\|_{W^{\alpha - \lambda}} \right)\max \left\{ t, t^{\frac{\alpha}{\lambda}} \right\}\right) \Bigg)\\
        \leq &\exp \Bigg( -N^{2} \beta \Bigg[ \frac{ \left(r- C \|f\|_{W^{\kappa}}\max \left\{ t, t^{\frac{\kappa}{\lambda}} \right\}  \right )^{2} }{\left\| h_{-\frac{1}{2}}(f) \right\|^{2}_{L^{2}}} - \\
        &\ \ \ \ \ \ \left(  C \left( \frac{t^{1-\frac{d}{\gamma}} }{N} +t \right)+C \left( 1+  \|\mu^{V}_{\infty}\|_{W^{\alpha - \lambda}} \right)\max \left\{ t, t^{\frac{\alpha}{\lambda}} \right\} \right)  \Bigg]  + N\left( {\rm ent}[\mu^{V}_{\infty}] - \log \left|\Sigma\right| + o_{N}(1) \right) \Bigg).
    \end{split}
\end{equation}
In equation \eqref{eq:reallylong}, we have used equation \eqref{eq:testerror1} to pass from the first to the second line, the fact that $ \left| {\rm Fluct}^{t}_{\infty}[f]   \right| \leq \left\| h_{-\frac{1}{2}}(f) \right\|_{L^{2}} \sqrt{\mathcal{E}\left({\rm emp}^{t}_{N} - \mu^{V}_{\infty}\right) } $ (which can be easily verified using Placherel Theorem and Cauchy-Schartz inequality) to pass from the second to the third line, Lemma \ref{lem:regularization} to pass from the fourth to the fifth line, and Lemma \ref{lem:firstconc} to pass from the fifth to the sixth line. 

In order to find the optimal $t$, we split into $4$ cases:
\begin{itemize}
    \item[]Case $1$: $\kappa \geq \lambda, \alpha \geq \lambda $. Then the optimal $t$ is given by $t = N^{-\frac{\gamma}{d}}$, and
    \begin{equation}
        \max \left\{ t, t^{\frac{\kappa}{\lambda}} \right\} +  C \left( \frac{t^{1-\frac{d}{\gamma}} }{N} +t \right) + \max \left\{ t, t^{\frac{\alpha}{\lambda}} \right\} \leq C N^{-\frac{\gamma}{d}}.
    \end{equation}

        \item[]Case $2$: $\kappa \leq \lambda, \alpha \geq \lambda $. Then the optimal $t$ is given by $t = N^{\left(1 - \frac{d}{\gamma} - \frac{\alpha}{\lambda}\right)^{-1}}$, and
    \begin{equation}
        \max \left\{ t, t^{\frac{\kappa}{\lambda}} \right\} +  C \left( \frac{t^{1-\frac{d}{\gamma}} }{N} +t \right) + \max\left\{ t, t^{\frac{\alpha}{\lambda}} \right\} \leq C N^{p^{*}(\alpha, \lambda, \gamma)}.
    \end{equation}

        \item[]Case $3$: $\kappa \geq \lambda, \alpha \leq \lambda $. Then the optimal $t$ is given by $t = N^{\left(1 - \frac{d}{\gamma} - \frac{\kappa}{\lambda}\right)^{-1}}$, and
    \begin{equation}
        \max \left\{ t, t^{\frac{\kappa}{\lambda}} \right\} +  C \left( \frac{t^{1-\frac{d}{\gamma}} }{N} +t \right) + \max\left\{ t, t^{\frac{\alpha}{\lambda}} \right\} \leq C N^{p^{*}(\kappa, \lambda, \gamma)}.
    \end{equation}

        \item[]Case $4$: $\kappa \leq \lambda, \alpha \leq \lambda $. Then the optimal $t$ is given by $t = N^{\left(1 - \frac{d}{\gamma} - \frac{\min\{\kappa, \alpha\}}{\lambda}\right)^{-1}}$, and we get  
    \begin{equation}
        \max \left\{ t, t^{\frac{\kappa}{\lambda}} \right\} +  C \left( \frac{t^{1-\frac{d}{\gamma}} }{N} +t \right) + \max\left\{ t, t^{\frac{\alpha}{\lambda}} \right\} \leq C N^{p^{*}(\min\{\alpha, \kappa\}, \lambda, \gamma)}.
    \end{equation}
\end{itemize}

From this, we can conclude equation \eqref{eq:conceq}. The proof of equation \eqref{eq:conctherm} is analogous. 

\textbf{Step 3:} Proof of lower bound for admissible kernels.

We start by proving equation \eqref{eq:lowbound1}. Let $\Lambda_{N}$ be as in Lemma \ref{lem:construction} with $\varphi=\mu^{V}_{\infty} + \left(r + C N^{- \alpha p }  \|f\|_{W^{\alpha}}\right) \frac{h_{-1}(f)}{\| h_{-\frac{1}{2}}(f) \|^{2}_{L^{2}}}$, where $C$ is the constant that appears in equation \eqref{eq:testerror} (note that $\varphi$ is positive for $N$ large enough in virtue of assumption $ii)$). Note that if $X_{N} \in \Lambda_{N}$ then
\begin{equation}
\label{eq:A}
    \begin{split}
        \frac{1}{N} \sum_{i=1}^{N} f(x_{i}) &\geq \int_{\mathbf{T}^{d}} f \varphi \, {\rm d} x - C N^{- \alpha p }  \|f\|_{W^{\alpha}} \\
        &\geq (r + C N^{- \alpha p }  \|f\|_{W^{\alpha}})\int_{\mathbf{T}^{d}} f  \frac{h_{-1}(f)}{\| h_{-\frac{1}{2}}(f) \|^{2}_{L^{2}}} \, {\rm d} x - C N^{- \alpha p }  \|f\|_{W^{\alpha}} +  \int_{\mathbf{T}^{d}} \mu^{V}_{\infty} f \, {\rm d} x \\
        &= r +  \int_{\mathbf{T}^{d}} \mu^{V}_{\infty} f \, {\rm d} x.
    \end{split}
\end{equation}
In equation \eqref{eq:A} we have used that $\int_{\mathbf{T}^{d}} f h_{-1}(f) \, {\rm d} x = \| h_{-\frac{1}{2}}(f) \|^{2}_{L^{2}}$.

Therefore for any $p > \frac{1}{d}$,
\begin{equation}
\label{eq:B}
    \begin{split}
         &\mathbf{P}_{N, \beta} \left( \left| {\rm Fluct}_{\infty}[f]  \right| \geq r \right) \\
         \geq &\mathbf{P}_{N, \beta} \left( X_{N} \in {\Lambda}_{N} \right)\\
         \geq & \frac{1}{K_{N, \beta}^{\infty}} \int_{\Lambda_{N}}  \exp \left( - N^{2} \beta \left[ {\rm F}_{N}(X_{N}, \mu^{V}_{\infty}) \right]  \right) \Pi_{i=1}^{N} \rho_{N}(x_{i}) \, \mathrm d X_{N}\\
         \geq & \exp \left( -N^{2}\beta \left[ \mathcal{E}\left( \varphi - \mu^{V}_{\infty} \right) - N^{p(s-d)} +C\left(1+\|\mu^{V}_{\infty}\|_{W^{\alpha - \gamma}} \right)N^{p^{*}}  \right]- o(N) \right)  p_{N, \beta}^{\otimes N} \left( \Lambda_{N} \right)\\
          \geq & \exp \Bigg( -N^{2}\beta \Bigg[ \left(r + C N^{- \alpha p }  \|f\|_{W^{\alpha}}\right)^{2} \| h_{-\frac{1}{2}}(f) \|_{L^{2}}^{-2} \\
          &\ \ \ \ \ \ \ \ +C\left(1+\|\mu^{V}_{\infty}\|_{W^{\alpha - \gamma}} \right)N^{\max \{p^{*}, p(s-d) \}}  \Bigg]- o(N) \Bigg)  p_{N, \beta}^{\otimes N} \left( \Lambda_{N} \right)\\
          \geq & \exp \Bigg( -N^{2}\beta \left[ \left(r + C N^{- \alpha p }  \|f\|_{W^{\alpha}}\right)^{2} \| h_{-\frac{1}{2}}(f) \|_{L^{2}}^{-2} +C\left(1+\|\mu^{V}_{\infty}\|_{W^{\alpha - \gamma}} \right)N^{\max \{p^{*}, p(s-d) \}} \right] \\
          &\ \ \ \ \ \ \ \ - N \left( \mathrm{ent}\left[\mu^{V}_{\infty} + r \frac{h_{-1}(f)}{\| h_{-\frac{1}{2}}(f) \|^{2}_{L^{2}}} \Bigg| \frac{1}{|\Sigma|}\mathbf{1}_{\Sigma}\right] + o(1) \right) \Bigg).
    \end{split}
\end{equation}
In equation \eqref{eq:B} we have used that $\mathcal{E}(h_{-1}(f)) =  \| h_{-\frac{1}{2}}(f) \|_{L^{2}}^{2}$. From this, we can conclude equation \eqref{eq:lowbound1}. The proof of equation \eqref{eq:lowbound2} is analogous and hence omitted.  
\end{proof}

\section{Proof of asymptotes of the Laplace transform}
\label{sec:laplace}

In this section, we show the proof of Theorem \ref{teo:laplace}. 

\begin{proof}

\textbf{Step 1:} Proof of asymptotes for weakly admissible kernels.

By \cite{garcia2019large}, Theorem 1.2, we have that 
\begin{equation}
        \lim_{N \to \infty} \frac{1}{N^{2} \beta } \log \left( \mathbb{E}_{\mathbf{P}_{N, \beta}} \left( \exp \left( N^{2} \beta r {\rm Fluct}_{\theta}[f]\right) \right) \right) = \inf_{\mu \in \mathcal{P}(\mathbf{T}^{d})} \mathcal{E}_{V}(\mu) + r \int_{\mathbf{T}^{d}} f \, \mathrm d (\mu - \mu_{\infty}^{V}) - \mathcal{E}_{V}(\mu_{\infty}^{V}).
    \end{equation}
The rest of the proof is analogous to step 1 of the proof of Theorems \ref{teo:concup} and \ref{teo:conclow}, with the additional easiness that in this case, we do not deal with a Lagrange multiplier.     

\textbf{Step 2:} Proof of item a).

This proof is inspired by the author's previous work \cite{garcia2022generalized}. We begin by applying Theorem \ref{teo:concup} and writing
\begin{equation}
    \begin{split}
        &\mathbb{E}_{\mathbf{P}_{N, \beta}} \left( \exp \left( N^{2} \beta r {\rm Fluct}_{\infty}[f]\right) \right)\\
        =& \int_{0}^{\infty}  {\mathbf{P}_{N, \beta}} \left( \exp \left( N^{2} \beta r {\rm Fluct}_{\infty}[f]\right) \geq x  \right) \, {\rm d} x\\
        =& \int_{0}^{\infty}  {\mathbf{P}_{N, \beta}} \left(    {\rm Fluct}_{\infty}[f] \geq \frac{\log x}{rN^{2} \beta}  \right) \, {\rm d} x\\
         \leq & \exp \left( N\left( {\rm ent}[\mu^{V}_{\infty}] - \log \left|\Sigma\right| + o_{N}(1) \right) - C N^{2 + q^{*}} \beta   \left( 1+ \|\mu^{V}_{\infty}\|_{W^{\alpha - \lambda}} \right) \right) \\
         &\ \ \ \ \ \ \int_{0}^{\infty}  \exp \left( -N^{2} \beta \left[ \frac{ \left(\frac{\log x}{r N^{2} \beta} - C N^{q^{*}}\|f\|_{W^{\kappa}} \right)^{2} }{\left\| h_{-\frac{1}{2}}(f) \right\|^{2}_{L^{2}}} \right] \right)\, {\rm d} x.
    \end{split}
\end{equation}

Recall that $ q^{*}:= p^{*}(\min\{\alpha, \kappa\}, \lambda, \gamma)$. Performing the change of variables $y = \log x $, we can transform the integral into 
\begin{equation}
\begin{split}
        &\mathbb{E}_{\mathbf{P}_{N, \beta}} \left( \exp \left( N^{2} \beta r {\rm Fluct}_{\infty}[f]\right) \right) \\
        &\leq \exp \left( N\left( {\rm ent}[\mu^{V}_{\infty}] - \log \left|\Sigma\right| + o_{N}(1) \right) - C N^{2 + q^{*}} \beta  \left( 1+   \|\mu^{V}_{\infty}\|_{W^{\alpha - \lambda}} \right) \right)  \\
        &\ \ \ \int_{0}^{\infty}  \exp \left( -N^{2} \beta \left[ \frac{ \left(\frac{y}{r N^{2} \beta} - C N^{q^{*}}\|f\|_{W^{\kappa}} \right)^{2} }{ \left\| h_{-\frac{1}{2}}(f) \right\|^{2}_{L^{2}}}  \right] - y \right)\, {\rm d} y.
\end{split}        
\end{equation}

This last integral is the density of a Gaussian random variable and can be computed explicitly using standard techniques. The result is 
\begin{equation}
\begin{split}
     &\int_{0}^{\infty}  \exp \left( -N^{2} \beta  \left[ \frac{ \left(\frac{y}{r N^{2} \beta} -  C N^{q^{*}}\|f\|_{W^{\kappa}} \right)^{2} }{\left\| h_{-\frac{1}{2}}(f) \right\|^{2}_{L^{2}}}  \right] - y \right)\, {\rm d} y \\
    =& \left( \frac{N^{2} \beta r^{2} \left\| h_{-\frac{1}{2}}(f) \right\|^{2}_{L^{2}}}{2} \right)^{\frac{d}{2}} \exp \left( N^{2} \beta \left[\frac{ r^{2} \left\| h_{-\frac{1}{2}}(f) \right\|^{2}_{L^{2}}}{4} -  C N^{q^{*}}\|f\|_{W^{\kappa}} r \right] \right) .
\end{split}    
\end{equation}

From this we can conclude. 

\textbf{Step 2:} Proof of item b).

First, note that
\begin{equation}
    \begin{split}
        &\mathbb{E}_{\mathbf{P}_{N, \beta}} \left( \exp \left( N^{2} \beta r {\rm Fluct}_{\infty}[f]\right) \right)\\
        =& \frac{1}{Z_{N, \beta}^{V}} \int_{\mathbf{T}^{d \times N}} \exp \left( N^{2} \beta r \int_{\mathbf{T}^{d}} f \, {\rm d} ({\rm emp}_{N} - \mu^{V}_{\infty}) \right) \exp \left( \beta \mathcal{H}_{N} X_{N} \right) \, {\rm d} X_{N}\\
        =& \frac{Z_{N, \beta}^{V- rf}}{Z_{N, \beta}^{V}} \exp \left( -N^{2} \beta r \int_{\mathbf{T}^{d}} f\mu^{V}_{\infty} \, {\rm d} x  \right).
    \end{split}
\end{equation}

Since $ \widehat{h_{-1}(f)} (0):= 0,$ we have that $\mu_{\infty}^{V} + r\frac{h_{-1}(f)}{2}$ has integral $1$. Furthermore, using the hypotheses of item b), we have that $\mu_{\infty}^{V} + r\frac{h_{-1}(f)}{2}$ is a probability measure that satisfies the First Order Condition, equation \eqref{ELeq}, and therefore $\mu_{\infty}^{V - r f} = \mu_{\infty}^{V} + r\frac{h_{-1}(f)}{2}$.

Using Lemma \ref{lem:partfunc}, we have that
\begin{equation}
\begin{split}
     \frac{Z_{N, \beta}^{V- r f}}{Z_{N, \beta}^{V}} \geq &\exp\Bigg(-N^{2} \beta \Bigg[ \mathcal{E}_{V -r f} (\mu_{\infty}^{V - r f}) - \mathcal{E}_{V} (\mu_{\infty}^{V}) \\
     & \ \ \ \ \ -C\|\mu^{V}_{\infty}\|_{W^{\alpha - \gamma}} N^{p^{*}}  - \frac{1}{\theta} \left( {\rm ent}[\mu^{V}_{\infty}] - \log |\Sigma| + o_{N}(1)  \right) \Bigg] \Bigg).
\end{split}     
\end{equation}

Using the splitting formula (Proposition \ref{Prop:splitting}), we have that 
\begin{equation}
    \begin{split}
         \mathcal{E}_{V - r f} (\mu_{\infty}^{V - r f}) &= \mathcal{E}_{V}(\mu_{\infty}^{V}) + \mathcal{E}\left( r\frac{h_{-1}(f)}{2} \right) -  r \int_{\mathbf{T}^{d}} f\mu^{V}_{\infty} \, {\rm d} x  - \frac{r^{2}}{2} \int_{\mathbf{T}^{d}} f h_{-1}(f)\, {\rm d} x \\
         &= \mathcal{E}_{V}(\mu_{\infty}^{V}) - \frac{r^{2}}{4} \| h_{-\frac{1}{2}}(f) \|_{L^{2}} - r \int_{\mathbf{T}^{d}} f\mu^{V}_{\infty} \, {\rm d} x.
    \end{split}
\end{equation}

From this, we can conclude. 

\end{proof}

\begin{remark}
    In principle, it is possible to apply the approach of step 2 (i.e. Expressing the Laplace transform in terms of partition functions, and using Lemma \ref{lem:partfunc} to estimate the partition functions) to obtain an upper bound on the Laplace transform as well. However, even if $f$ satisfies hypotheses $i)$ and $ii)$, the resulting bound is
     \begin{equation}
    \begin{split}
   \log \left(  \mathbb{E}_{\mathbf{P}_{N, \beta}} \left( \exp \left( N^{2} \beta r {\rm Fluct}_{\infty}[f]\right) \right) \right) \leq   N^{2} \beta \Bigg[ \frac{r^{2}}{4} \| h_{-\frac{1}{2}}(f) \|_{L^{2}}^{2} \\
    +C\|\mu^{V-rf}_{\infty}\|_{W^{\alpha - \gamma}} N^{p^{*}}  + \frac{1}{\theta} \left( {\rm ent}[\mu^{V-rf}_{\infty}] - \log |\Sigma| + o_{N}(1)  \right)  \Bigg], 
   \end{split}
    \end{equation}
    yielding an error that is no better than the one obtained in Theorem \ref{teo:laplace}. The general case (if $f$ does not satisfy hypotheses $i)$ and $ii)$) is more complicated, since it is not possible to easily characterize $\mu^{V-rf}_{\infty}$. 
\end{remark}

\section{Appendix A: Regularization procedure}
\label{sect:regularization}

In this section, we introduce a regularization procedure for general kernels. The idea to regularize is to look at a kernel-dependent parabolic equation with the empirical measure as its initial condition and solve for a small time. This is inspired by the approach of \cite{garcia2019concentration}, which applies an analysis of the heat equation in a general compact manifold to the study of one-component plasma. 

One of the advantages of this approach is that we can omit a restrictive hypothesis on the interaction $g$, namely the existence of a higher-dimensional superharmonic extension, which was present in \cite{rosenzweig2023global, nguyen2022mean, garcia2022generalized}. This approach also bypasses the need to work in higher dimensions.

We now prove Lemma \ref{lem:regularization}, restated here:

\begin{lemma}
Let $N>1$, $X_{N} \in \mathbf{T}^{d \times N}$, $g: \mathbf{T}^{d}$ be admissible, and $\alpha>0$. Then for all $t>0$ there exists ${\rm emp}_{N}^{t} \in \mathcal{P}(\mathbf{T}^{d})$ such that
\begin{equation}
        \mathcal{E}\left( {\rm emp}_{N}^{t} \right) \leq  \mathcal{E}^{\neq}\left( {\rm emp}_{N} \right) + C \left( \frac{t^{1-\frac{d}{\gamma}} }{N} +t \right),
\end{equation}
\begin{equation}
       \left| \int_{\mathbf{T}^{d}} f \, \mathrm{d} \left( {\rm emp}_{N}^{t} - {\rm emp}_{N} \right) \right| \leq C \max \left\{ t, t^{\frac{\alpha}{\lambda}} \right\} \|f\|_{W^{\alpha}},
\end{equation}
for some $C>0$ depending only on $g$. We call ${\rm emp}_{N}^{t} $ the {regularized empirical measure}. 
\end{lemma}

\begin{remark}
    If $g$ is a Riesz kernel of order $s$, then by taking $t = N^{-\frac{2s}{d}}$, we obtain the estimate
    \begin{equation}
        \mathcal{E}\left( {\rm emp}_{N}^{t} \right) \leq \mathcal{E}^{\neq}\left( {\rm emp}_{N} \right) + C N^{-\frac{2s}{d}},
\end{equation}
which is optimal in terms of order of magnitude. 
\end{remark}

\begin{proof}
    \textbf{Step 1:} Definition.

    We introduce the notation $L= h_{-1}$. We will be interested in the following parabolic equation:
    \begin{equation}
    \label{eq:heatflow}
        \begin{split}
            \partial_{t}u + L(u) &=0\\
            u(\cdot, 0)&= u_{0}. 
        \end{split}
    \end{equation}

    Note that $u$ can be expressed in terms of the Fourier coefficients of $g$ and $u_{0}$: 
    \begin{equation}
    \label{eq:solution}
        u (x, t) = \widehat{u_{0}}(0) + \sum_{m \in \mathbf{Z}^{d} \setminus \{0\}} \exp \left( - \widehat{g}^{-1}(m) t \right) \widehat{u_{0}}(m) \exp\left( \frac{2 \pi i}{T} m \cdot x\right),
    \end{equation}
    since, for $m \neq 0$, 
    \begin{equation}
        L \left(\exp\left( \frac{2 \pi i}{T} m \cdot x\right) \right) = \widehat{g}^{-1}(m) \exp\left( \frac{2 \pi i}{T} m \cdot x\right).
    \end{equation}
    In other words, the solution can be written as 
    \begin{equation}
        u(\cdot, t) = u_{0} \ast p(\cdot, t),
    \end{equation}
    where
    \begin{equation}
    \label{eq:kernelp}
        p(x,t) := 1 + \sum_{m \in \mathbf{Z}^{d} \setminus \{0\}} \exp \left( - \widehat{g}^{-1}(m) t \right) \exp\left( \frac{2 \pi i}{T} m \cdot x\right).
    \end{equation}
    Note that $u$ given by equation \eqref{eq:solution} is smooth for all $t>0$ if $u_{0}$ is a measure of bounded variation, and hence $u$ solves equation \eqref{eq:heatflow} in the classical sense. 
    
    The regularization ${\rm emp}_{N}^{t}$ is now defined as the solution to equation \eqref{eq:heatflow} with initial condition $u_{0} = {\rm emp}_{N}$. 

    \textbf{Step 2:} Energy estimate.

    We proceed to estimate the energy of ${\rm emp}_{N}^{t}$. We begin with the elementary calculation
    \begin{equation}
        \mathcal{E}\left( {\rm emp}_{N}^{t} \right) -  \mathcal{E}^{\neq}\left( {\rm emp}_{N} \right) = \frac{1}{N}\mathcal{E}\left( \delta_{x}^{t} \right) + \frac{1}{N^{2}} \sum_{ i \neq j} g(x_{i} - x_{j}) - \mathcal{G}\left( \delta_{x}^{t}, \delta_{y}^{t} \right),
    \end{equation}
    where $\delta_{x}^{t}, \delta_{y}^{t}$ are, respectively, the solutions to equation \eqref{eq:heatflow} with initial condition $u_{0} = \delta_{x}$ or $u_{0} = \delta_{y}$. Hence, the rest of the proof will be about estimating $\mathcal{G}\left( \delta_{x}^{t}, \delta_{y}^{t} \right)$ and $\mathcal{E}\left( \delta_{x}^{t} \right)$. First, note that using the symmetry of the kernel $g$, we have that $\mathcal{G}\left( \delta_{x}^{t}, \delta_{y}^{t} \right) = g_{2t}(x-y)$, where $g_{t}$ is the solution to equation \eqref{eq:heatflow} with initial condition $u_{0} = g$. In particular, $\mathcal{E}\left( \delta_{x}^{t} \right) = g_{2t}(0)$. 
    
    \textit{Substep 2.1:} Starting point.

    We will first find a suitable expression for $g_{2t}(z)$, for arbitrary $z \in \mathbf{T}^{d}$. We begin by writing
    \begin{equation}
    \label{eq:integralform}
        \begin{split}
            g_{2t}(z) &= \widehat{g}(0) + \sum_{m \in \mathbf{Z}^{d} \setminus \{0\}} \exp \left( - 2 \widehat{g}^{-1}(m) t \right) \widehat{g}(m) \exp\left( \frac{2 \pi i}{T} m \cdot z \right)\\
            &=\widehat{g}(0) + \sum_{m \in \mathbf{Z}^{d}\setminus \{0\}} \int_{2t}^{\infty} \exp \left(- \widehat{g}^{-1}(m) s \right) \exp\left( \frac{2 \pi i}{T} m \cdot z \right) \, \mathrm d s\\
            &=\widehat{g}(0) + \int_{2t}^{\infty} \sum_{m \in \mathbf{Z}^{d}\setminus \{0\}}  \exp \left(- \widehat{g}^{-1}(m) s \right) \exp\left( \frac{2 \pi i}{T} m \cdot z \right) \, \mathrm d s\\
            &=\widehat{g}(0) + \int_{2t}^{\infty} p(z, s) -1 \, \mathrm d s.
        \end{split}
    \end{equation}

    We will now use equation \eqref{eq:integralform} to derive estimates for $g_{2t}(0)$ and $g(z)-g_{2t}(z)$ for arbitrary $z$. 
    
    \textit{Substep 2.2:} Estimate for $g_{2t}(0)$.  

    We begin by writing
    \begin{equation}
        p(0,s) = 1+  \sum_{m \in \mathbf{Z}^{d}\setminus \{0\}}  \exp \left( -\widehat{g}^{-1}(m) s \right).
    \end{equation}
    Using assumption 2 in Definition \ref{def:admissible}, we have that, for some $C>0$ and $s$ small enough,
    \begin{equation}
        \begin{split}
            p(0,s) &\leq 1+ \sum_{m \in \mathbf{Z}^{d}}  \exp \left( -C |m|^{\gamma} s \right)\\
            &\leq 1+ C \int_{\mathbf{R}^{d}}  \exp \left( -C |x|^{\gamma} s \right) \, \mathrm d x\\
            &\leq 1+ C s^{-\frac{d}{\gamma}} \int_{\mathbf{R}^{d}}  \exp \left( -C |y|^{\gamma} \right) \, \mathrm d y\\
            &\leq C s^{-\frac{d}{\gamma}},
        \end{split}
    \end{equation}

    since $\exp\left( -C |y|^{\gamma} \right)$ is an integrable function. We then have that, for $t$ small enough,
    \begin{equation}
        \begin{split}
            g_{2t}(0) &= \widehat{g}(0) + \int_{2t}^{1} p(0,s) -1 \, \mathrm d s + \int_{1}^{\infty} p(0,s) -1 \, \mathrm d s\\
            &= \widehat{g}(0) +\int_{2t}^{1} p(0,s) -1 \, \mathrm d s + g_{1}(0)\\
            &\leq C t^{-\frac{d}{\gamma} + 1}.
        \end{split}
    \end{equation}
    
    \textit{Substep 2.3:} Estimate for $g(z) - g_{2t}(z)$ and conclusion of step 2. 

    Using equation \eqref{eq:integralform}, we have that
    \begin{equation}
        g(z) - g_{2t}(z) = \int_{0}^{2t} p(z, s) - 1 \, \mathrm d s.
    \end{equation}

    Using assumption 4 of Definition \ref{def:admissible}, we have that, for small enough $t$,
    \begin{equation}
        g(z) - g_{2t}(z) \geq Ct .
    \end{equation}

    Putting everything together, we have the estimate
    \begin{equation}
        \begin{split}
            \mathcal{E}\left( {\rm emp}_{N}^{t} \right) - \mathcal{E}^{\neq}\left( {\rm emp}_{N} \right) &= \frac{1}{N}\mathcal{E}\left( \delta_{0}^{t} \right) + \frac{1}{N^{2}} \sum_{ i \neq j} g(x_{i} - x_{j}) - g_{2t}(x_{i} - x_{j}) \\
            &\leq C \left( \frac{t^{-\frac{d}{\gamma}+1}}{N} + t \right).
        \end{split}
    \end{equation}

    \textbf{Step 3:} Test error.

    Once again, using the symmetry of the kernel $g$, we have
    \begin{equation}
    \label{eq:intdisc}
        \begin{split}
            \left| \int_{\mathbf{T}^{d}} f \, \mathrm d \left( {\rm emp}_{N} - {\rm emp}_{N}^{t}\right) \right| &= \left| \int_{\mathbf{T}^{d}} f - f_{t} \, \mathrm d  {\rm emp}_{N} \right|\\
            &\leq \| f - f_{t} \|_{L^{\infty}}\\
            &\leq \sum_{m \in \mathbf{Z}^{d} \setminus \{0\}} \left| \left( 1 - \exp \left(-\widehat{g}^{-1}(m) t \right) \right) \widehat{f}(m) \right|,
        \end{split}
    \end{equation}
    where $f_{t}$ denotes the solution to equation \eqref{eq:heatflow} with $u_{0} = f$. 
    The rest of this step will consist in estimating the last line of equation \eqref{eq:intdisc}. The idea will be to use the bound 
    \begin{equation}
        \left| 1 - \exp \left(-\widehat{g}^{-1}(m) t \right)\right| \leq C \widehat{g}^{-1}(m) t
    \end{equation}
    for $t$ small, i.e. Smaller than $M_{*}$ and the bound 
    \begin{equation}
       \left| 1 - \exp \left(-\widehat{g}^{-1}(m) t \right)\right| \leq 1   
    \end{equation}
    for $t$ large, i.e. Larger than $M_{*}$, and then optimize over $M_{*}$. With this in mind, we let $M_{*}>0$ to be determined later and using assumption 3 of Definition \ref{def:admissible} we have
     \begin{equation}
        \begin{split}
            \left| \int_{\mathbf{T}^{d}} f \, \mathrm d \left( {\rm emp}_{N} - {\rm emp}_{N}^{t}\right) \right| &\leq C \left( \sum_{0<|m| \leq M_{*}} \left| \widehat{g}^{-1}(m) t \widehat{f}(m)\right| + \sum_{|m| \geq M_{*}} \left| \widehat{f}(m)\right| \right)\\
            &\leq  C \left( \sum_{0<|m| \leq M_{*}} \left| |m|^{\lambda - \alpha} t |m|^{\alpha} \widehat{f}(m)\right| + \sum_{|m| \geq M_{*}} \left| \widehat{f}(m)\right| \right).
        \end{split}
    \end{equation}

    Note that 
    \begin{equation}
        \begin{split}
            \sum_{|m| \geq M_{*}} \left| \widehat{f}(m)\right| &\leq M_{*}^{-\alpha}\sum_{|m| \geq M_{*}} \left| |m|^{\alpha} \widehat{f}(m)\right| \\
            &\leq M_{*}^{-\alpha} \| f \|_{W^{\alpha}}.
        \end{split}
    \end{equation}

    We split the remainder of the proof into two cases: $\lambda - \alpha \geq 0$ and $\lambda - \alpha \leq 0$.

    \begin{itemize}
        \item[]Case 1: $\lambda - \alpha \geq 0$.

        In this case, we have that 
        \begin{equation}
            \sum_{0<|m| \leq M_{*}} \left| |m|^{\lambda - \alpha} t |m|^{\alpha} \widehat{f}(m)\right| \leq M_{*}^{\lambda - \alpha} t \|f\|_{W^{\alpha}},
        \end{equation}
        and consequently,
        \begin{equation}
            \left| \int_{\mathbf{T}^{d}} f \, \mathrm d \left( {\rm emp}_{N} - {\rm emp}_{N}^{t}\right) \right| \leq C \|f\|_{W^{\alpha}} \left(  M_{*}^{\lambda - \alpha} t + M_{*}^{-\alpha} \right).
        \end{equation}
        Optimizing $M_{*}$ yields $M_{*}^{\lambda - \alpha} t = M_{*}^{-\alpha}$, i.e. $M_{*} = t^{-\frac{1}{\lambda}}$, and the bound 
        \begin{equation}
            \left| \int_{\mathbf{T}^{d}} f \, \mathrm d \left( {\rm emp}_{N} - {\rm emp}_{N}^{t}\right) \right| \leq C \|f\|_{W^{\alpha}} t^{\frac{\alpha}{\lambda}}.
        \end{equation}    

        \item[]Case 2: $\lambda - \alpha \leq 0$.

        In this case, we have that 
        \begin{equation}
            \sum_{0<|m| \leq M_{*}} \left| |m|^{\lambda - \alpha} t |m|^{\alpha} \widehat{f}(m)\right| \leq t \|f\|_{W^{\alpha}},
        \end{equation}
        and consequently,
        \begin{equation}
            \left| \int_{\mathbf{T}^{d}} f \, \mathrm d \left( {\rm emp}_{N} - {\rm emp}_{N}^{t}\right) \right| \leq C \|f\|_{W^{\alpha}} \left( t + M_{*}^{-\alpha} \right).
        \end{equation}
        Optimizing $M_{*}$ yields $M_{*} = t^{-\frac{1}{\alpha}}$, and the bound 
        \begin{equation}
            \left| \int_{\mathbf{T}^{d}} f \, \mathrm d \left( {\rm emp}_{N} - {\rm emp}_{N}^{t}\right) \right| \leq C \|f\|_{W^{\alpha}} t.
        \end{equation}
    \end{itemize}

    From this, we can conclude. 
\end{proof}

\begin{proposition}[Sufficient conditions for assumption 4]
    \label{prop:sufficient}
    Assumption 4 of Definition \ref{def:admissible} holds if $g$ is the periodic Coulomb kernel. Assumption 4 also holds if $L$ satisfies a comparison principle, i.e. if $L$ satisfies that, for any $f$, if $x$ is a minimum of $u$, then $L(u)(x) \geq 0$; furthermore if $x$ is a minimum of $u$, and $L(u)(x) = 0$, then $u$ is constant.
\end{proposition}

\begin{remark}
    In particular, assumption 4 holds for the periodic Riesz and Riesz-type kernels (see \cite{garcia2022generalized, nguyen2022mean}). Assumption 4 also holds if $h := \widehat{\left(\frac{1}{\widehat{g}}\right)}$ satisfies that $h >0$ a.e. and $h(x) \leq C |x|^{s}$ for some $s<2$, since in that case $L$ can be written as 
    \begin{equation}\label{eq:inverse}
    L(u) (x) =  \frac{1}{2} \int_{\mathbf{R}^{d}} \left( -2u(x)+u(x+y) + u(x-y) \right)  h(y)  \, \mathrm d y.
    \end{equation}
\end{remark}

\begin{proof}[Proof of Proposition \ref{prop:sufficient}]
    If $g$ is the Coulomb kernel, then $L = \Delta$, and $p$ is the heat kernel, which is well-known to be positive. 

    If $L$ satisfies a comparison principle, we will prove that equation \eqref{eq:heatflow} satisfies that the minimum of $u(x,t)$ is achieved at $t=0$. To show this claim, we proceed by contrapositive. Note that, for any $u_{0}$ of bounded variation, assumption 2 of Definition \ref{def:admissible} implies that $u(x,t)$ is smooth for any $t>0$. Note also that $\lim_{t \to \infty} u(x,t) = \fint u_{0}$. In particular, if the minimum of $u$ is not achieved at $t=0$, then the minimum of $u(x,t)$ is achieved at some $(x^{*}, t^{*}) \in \mathbf{T}^{d} \times \mathbf{R}^{+}$, and $\partial_{t}u(x^{*}, t^{*}) = 0$. Since $u$ satisfies equation \eqref{eq:heatflow}, $L(u)(x^{*}, t^{*})=0$, which implies that $u(\cdot, t^{*})$ is constant by the comparison principle. Since $u \mapsto p(\cdot, s) \ast u$ is invertible, $u_{0}$ is constant. 

    Therefore, if $u_{0}$ is positive a.e., then $u_{0} \ast p(\cdot, s)$ is positive a.e. Therefore $p(\cdot, s)$ is positive almost everywhere, and assumption 4 is satisfied with $\epsilon = \infty$ and $C = 0$. 
\end{proof}

\section{Appendix B: Construction of configurations}
\label{sec:app2}

In this appendix, we prove Lemma \ref{lem:construction}, restated here:
\begin{lemma}
Let $\varphi$ be a probability density such that $\supp{\varphi} \subset \Sigma$, and let $f$ be a continuous function, and $g:\mathbf{T} \to \mathbf{R}$ be an admissible kernel. Then for any $p < \frac{1}{d}$ and $\tau >0$ there exists a family of configurations $\Lambda_{N} \subset \mathbf{T}^{d \times N}$ such that
\begin{itemize}
    \item[1.] For any $X_{N} \in \Lambda_{N}$,
    \begin{equation}
        \left| {\rm F}_{N}(X_{N}, \mu^{V}_{\infty}) - \mathcal{E}(\varphi - \mu^{V}_{\infty}) \right| \leq N^{p(s-d)}.
    \end{equation}
    
    \item[2.] \begin{equation}
        \log \left( p_{N, \beta}^{\otimes N} \left( \Lambda_{N} \right) \right) \geq -N \left( {\rm ent}[\varphi| p_{N, \beta}] + \tau \right).
    \end{equation}
    
    \item[3.] For any $X_{N} \in \Lambda_{N}$ and $\alpha>0$,
    \begin{equation}
       \left| \frac{1}{N} \sum_{i=1}^{N} f(x_{i}) - \int_{\mathbf{T}^{d}} f \varphi \, {\rm d} x \right| \leq 
        C N^{- \alpha p }  \|f\|_{W^{\alpha}}.
    \end{equation}
\end{itemize}
\end{lemma}

\begin{proof}
\textbf{Step 1:} Construction 

First, we subdivide $\Omega$ into cubes $K_{j}$ of size $\overline{\eta}>0$ and center $x_{j},$ for $\overline{\eta}>0$ to be determined later. 

Let either
\begin{equation}
    n_{j} = \ceil*{n \varphi(K_{j})}\footnote{We recall the abuse of notation of not distinguishing between a measure and its density.}
\end{equation}
or 
\begin{equation}
    n_{j} = \floor*{n \varphi(K_{j})},
\end{equation}
chosen so that 
\begin{equation}
    \sum_{j} n_{j} = n.
\end{equation}

The procedure for determining the point configuration of $n_{j}$ points is: $y_{1}$ is chosen at random from $K_{j}^{\tau},$ where $K_{j}^{\tau}$ is the cube $K_{j}$ minus a boundary layer of width $\tau$, $y_{2}$ is chosen at random from 
\begin{equation}
    K_{j}^{\tau} \setminus B(y_{1}, \tau).
\end{equation}
Then, for $i = 1...n_{j},$ the point $y_{i}$ is chosen at random from
\begin{equation}
    K_{j}^{\tau} \setminus \bigcup_{l=1}^{i-1} B(y_{l}, \tau).
\end{equation}

In other words, 
\begin{equation}
     \Lambda_{\delta}^{\eta, \epsilon} = \bigcup_{\sigma \in \text{sym}[1:N]} \bigotimes_{j} \bigotimes_{i=1}^{n_{j}} \left( K_{j}^{\tau} \setminus \bigcup_{ l=1}^{i-1} B(y_{\sigma(l)}, \tau) \right).
\end{equation}

We set $\tau = a \overline{\eta} n_{j}^{-\frac{1}{d}},$ for some $a \in (0,1)$ to be determined later. For $a$ small enough, the procedure is well defined, in the sense that it is possible to choose $n_{j}$ points in this way. 

We immediately get that $d(x_{i}, \partial \Omega) > r N^{-\frac{1}{d}}, d(x_{i}, x_{j})> r N^{-\frac{1}{d}}$ for some $r>0$. We now prove that this configuration has the right volume and energy.

\textbf{Step 2:} Test error.  

Note for any $x_{1}, x_{2} \in K_{j}$, we have that
\begin{equation}
     \left| f(x_{1}) - f(x_{2}) \right| \leq C \overline{\eta}^{\alpha} \|f\|_{C^{0,\alpha}},
\end{equation}
which implies that     
\begin{equation}
       \left| \frac{1}{N} \sum_{i=1}^{N} f(x_{i}) - \int_{\mathbf{T}^{d}} f \varphi \, {\rm d} x \right| \leq C \overline{\eta}^{\alpha} \|f\|_{C^{0,\alpha}}.
\end{equation}
On the other hand, it is easy to check (see Remark \ref{rem:Holder}) that
\begin{equation}
    \|f\|_{C^{0,\alpha}} \leq C \|f\|_{W^{\alpha}}. 
\end{equation}
From this we can conclude.

\textbf{Step 3:} Volume estimate 

Note that $p_{N, \beta}$ converges to the uniform probability measure on $\Sigma$, which we denote $U$. We begin with a direct computation:
\begin{equation}
    \begin{split}
        U^{\otimes N} (\Lambda_{\delta}^{\eta, \epsilon} ) &= \frac{N!}{\Pi_{i} n_{i}!} \Pi_{j} \Pi_{p=1}^{n_{j}-1} (\overline{\eta}^{d} - k_{d}\overline{\eta}^{d-1}\tau - c_{d} p \tau^{d}) \\
        &= \frac{N!}{\Pi_{i} n_{i}!} \Pi_{j} \overline{\eta}^{d n_{j}} \Pi_{p=1}^{n_{j}-1} (1 - \frac{\tau k_{d}}{\overline{\eta}} - \frac{c_{d} p a^{d}}{n_{j}}), 
    \end{split}
\end{equation}
where $c_{d}, k_{d}$ are constants which depend only on $d$. The first term being subtracted accounts for the volume loss due to the boundary layer, while the second one accounts for the volume loss due to excluding spheres around points. On the other hand, the volume of all configurations with exactly $n_{j}$ points in cube $K_{j}$ is given by
\begin{equation}
   \frac{N!}{\Pi_{i} n_{i}!} \Pi_{j} \overline{\eta}^{d n_{j}}.
\end{equation}

Using Stirling's approximation for the factorial, we have that 
\begin{equation}
    \frac{N!}{\Pi_{i} n_{i}!} \Pi_{j} \overline{\eta}^{d n_{j}} = \exp(-N [{\rm ent}[\varphi| \mu_{N}] + o_{N}(1)]).
\end{equation}

Furthermore, we may estimate the volume loss as follows:
\begin{equation}
    \begin{split}
        \log \left( \Pi_{j}  \Pi_{p=1}^{n_{j}-1} (1 - \frac{ k_{d} \tau}{\overline{\eta}} - \frac{c_{d} p a^{d}}{n_{j}})   \right) &=  \sum_{j}  \sum_{p=1}^{n_{j}-1} \log \left(1 - \frac{ k_{d}\tau}{\overline{\eta}} - \frac{c_{d} p a^{d}}{n_{j}}\right) \\
        &\leq  a  k_{d} \sum_{j} n_{j}^{1-\frac{1}{d}} + c_{d} a^{d} \sum_{j} n_{j} \\
        & \leq C a N.
    \end{split}
\end{equation}

By making $a$ small enough, we may conclude. 

\textbf{Step 4:} Energy estimate. 

\textit{Substep 4.1:} Starting point

Given $\epsilon>0$, we denote by $\delta_{\epsilon}$ the uniform probability measure on a sphere of radius $\epsilon$. We will also use the notation ${\rm emp}_{N}^{\epsilon} := {\rm emp}_{N}^{\epsilon} \ast \delta_{\epsilon}$. Using triangle inequality, we have that
\begin{equation}
\label{eq:startingpt}
        \left| {\rm F}_{N}(X_{N}, \mu^{V}_{\infty}) - \mathcal{E}(\varphi - \mu^{V}_{\infty}) \right| \leq \left| {\rm F}_{N}(X_{N}, \mu^{V}_{\infty}) - \mathcal{E}( {\rm emp}_{N}^{\epsilon} - \mu^{V}_{\infty}) \right| + \left| \mathcal{E}( {\rm emp}_{N}^{\epsilon} - \mu^{V}_{\infty}) - \mathcal{E}(\varphi - \mu^{V}_{\infty}) \right|. 
\end{equation}

We will now get an estimate for each of the terms on the RHS of equation \eqref{eq:startingpt}. 

\textit{Substep 4.2:} Estimate for $\left| {\rm F}_{N}(X_{N}, \mu^{V}_{\infty}) - \mathcal{E}( {\rm emp}_{N}^{\epsilon} - \mu^{V}_{\infty}) \right| $.

We begin by writing
\begin{equation}
    \begin{split}
            &\left|  {\rm F}_{N}(X_{N}, \mu^{V}_{\infty}) -  \mathcal{E}( {\rm emp}_{N}^{\epsilon} - \mu^{V}_{\infty}) \right| =\\
            &\left| \mathcal{E}^{\neq} \left( {\rm emp}_{N} \right) - \mathcal{E} \left( {\rm emp}_{N}^{\epsilon} \right) + 2 \mathcal{G} \left( {\rm emp}_{N} - {\rm emp}_{N}^{\epsilon}, \mu^{V}_{\infty} \right) \right| \leq \\
            & \frac{1}{N} \mathcal{E}(\delta^{\epsilon}) + \left|  \sum_{i \neq j} \int_{\mathbf{T}^{d}} h^{\delta_{x_{i}} - \delta_{x_{i}}^{\epsilon}} \delta_{x_{j}} \right|+ \left|  \sum_{i \neq j} \int_{\mathbf{T}^{d}} h^{\delta_{x_{i}} - \delta_{x_{i}}^{\epsilon}} \delta_{x_{j}}^{\epsilon} \right|+ 2\left|  \mathcal{G} \left( {\rm emp}_{N} - {\rm emp}_{N}^{\epsilon}, \mu^{V}_{\infty} \right)\right|.
    \end{split}
\end{equation}

We first derive an estimate for $2\left|  \mathcal{G} \left( {\rm emp}_{N} - {\rm emp}_{N}^{\epsilon}, \mu^{V}_{\infty} \right)\right|$ in the case of assumption 5 a) of Definition \ref{def:admissible}. To this end, note that for any $x \in \mathbf{T}^{d}$, and any $R>0$ small enough,
\begin{equation}
    \begin{split}
        \int_{\mathbf{T}^{d}} \left(g(x) - g_{\epsilon}(x) \right) \mu_{\infty}^{V}(x) &= \int_{B(0, R)} \left(g(x) - g_{\epsilon}(x) \right) \mu_{\infty}^{V}(x) + \int_{\Sigma \setminus B(0, R)} \left(g(x) - g_{\epsilon}(x) \right) \mu_{\infty}^{V}(x) \\
        &\leq \left| \mu_{\infty}^{V} \right|_{L^{\infty}} \left( \int_{B(0, R)} \left(g(x) - g_{\epsilon}(x) \right) + \int_{\Sigma \setminus B(0, R)} \left(g(x) - g_{\epsilon}(x) \right) \right)\\
        &\leq C \left| \mu_{\infty}^{V} \right|_{L^{\infty}} \left( \int_{B(0, R + \epsilon)} \frac{1}{|x|^{s}} + \int_{\Sigma \setminus B(0, R)} \frac{\epsilon}{|x|^{s +1}} \right)\\
        &\leq C \left| \mu_{\infty}^{V} \right|_{L^{\infty}} \left( \left( R+\epsilon\right)^{ d - s} + \epsilon R^{d - s -1}  \right).
    \end{split}
\end{equation}

The optimal $R$ is given by $R = \epsilon$ and the result is that
\begin{equation}
    \left|  \mathcal{G} \left( {\rm emp}_{N} - {\rm emp}_{N}^{\epsilon}, \mu^{V}_{\infty} \right)\right| \leq C \left| \mu_{\infty}^{V} \right|_{L^{\infty}} \epsilon^{d - s}. 
\end{equation}

Proceeding similarly, we may prove that 
\begin{equation}
    \left|  \sum_{i \neq j} \int h^{\delta_{x_{i}} - \delta_{x_{i}}^{\epsilon}} \delta_{x_{j}} \right|+ \left|  \sum_{i \neq j} \int h^{\delta_{x_{i}} - \delta_{x_{i}}^{\epsilon}} \delta_{x_{j}}^{\epsilon} \right| \leq  C \left| \varphi \right|_{L^{\infty}} \epsilon^{d - s}. 
\end{equation}

On the other hand,
\begin{equation}
    \begin{split}
        \frac{1}{N} \mathcal{E}(\delta^{\epsilon}) &\leq \frac{1}{N} \int_{B(0, \epsilon)} \int_{B(0, \epsilon)} \frac{1}{|x-y|^{s}} \, \mathrm d x \, \mathrm d y \\
        &= \frac{\epsilon^{-s}}{N}.
    \end{split}
\end{equation}

Putting everything together, we get that 
\begin{equation}
    \left| {\rm F}_{N}(X_{N}, \mu^{V}_{\infty}) - \mathcal{E}( {\rm emp}_{N}^{\epsilon} - \mu^{V}_{\infty}) \right| \leq C \left( \left| \varphi \right|_{L^{\infty}} + \left| \mu_{\infty}^{V} \right|_{L^{\infty}} \right) \epsilon^{d - s} + \frac{\epsilon^{-s}}{N}.
\end{equation}
The optimal bound is achieved at $\epsilon = N^{-\frac{1}{d}}$ and the result is 
\begin{equation}
     \left| {\rm F}_{N}(X_{N}, \mu^{V}_{\infty}) - \mathcal{E}( {\rm emp}_{N}^{\epsilon} - \mu^{V}_{\infty}) \right| \leq C \left( \left| \varphi \right|_{L^{\infty}} + \left| \mu_{\infty}^{V} \right|_{L^{\infty}} \right) N^{\frac{s - d}{d}}.
\end{equation}

The procedure to estimate the term $2\left|  \mathcal{G} \left( {\rm emp}_{N} - {\rm emp}_{N}^{\epsilon}, \mu^{V}_{\infty} \right)\right|$ in the case of assumption 5 b) of Definition \ref{def:admissible} is analogous. In this case,
\begin{equation}
    \begin{split}
        \int_{\mathbf{T}^{d}} \left(g(x) - g_{\epsilon}(x) \right) \mu_{\infty}^{V}(x) &\leq C \left| \mu_{\infty}^{V} \right|_{L^{\infty}} \left( -\int_{B(0, R + \epsilon)} \log|x| + \int_{\Sigma \setminus B(0, R)} \frac{\epsilon}{|x|} \right)\\
        &\leq C \left| \mu_{\infty}^{V} \right|_{L^{\infty}} \left( \left( R+\epsilon\right)^{ d } \log  \left( R+\epsilon\right) + \epsilon R^{d -1}  \right).
    \end{split}
\end{equation}

Taking $R = \epsilon$, the result is that
\begin{equation}
    \left|  \mathcal{G} \left( {\rm emp}_{N} - {\rm emp}_{N}^{\epsilon}, \mu^{V}_{\infty} \right)\right| \leq C \left| \mu_{\infty}^{V} \right|_{L^{\infty}} \epsilon^{d} |\log \epsilon|,  
\end{equation}
and similarly, 
\begin{equation}
    \left|  \sum_{i \neq j} \int h^{\delta_{x_{i}} - \delta_{x_{i}}^{\epsilon}} \delta_{x_{j}} \right|+ \left|  \sum_{i \neq j} \int h^{\delta_{x_{i}} - \delta_{x_{i}}^{\epsilon}} \delta_{x_{j}}^{\epsilon} \right| \leq  C \left| \varphi \right|_{L^{\infty}} \epsilon^{d} |\log \epsilon|. 
\end{equation}
On the other hand,
\begin{equation}
        \frac{1}{N} \mathcal{E}(\delta^{\epsilon}) \leq C \frac{\left|\log \epsilon \right|}{N}.
\end{equation}

Taking $\epsilon = N^{-\frac{1}{d}}$ and putting everything together, we have that 
\begin{equation}
     \left| {\rm F}_{N}(X_{N}, \mu^{V}_{\infty}) - \mathcal{E}( {\rm emp}_{N}^{\epsilon} - \mu^{V}_{\infty}) \right| \leq C \left( \left| \varphi \right|_{L^{\infty}} + \left| \mu_{\infty}^{V} \right|_{L^{\infty}} \right) \frac{\log N}{N}.
\end{equation}

\textit{Substep 4.3:} Estimate for $\left| \mathcal{E}( {\rm emp}_{N}^{\epsilon} - \mu^{V}_{\infty}) - \mathcal{E}(\varphi - \mu^{V}_{\infty}) \right|$.

Using polar decomposition for the quadratic form $\mathcal{E}(\cdot)$, we get
\begin{equation}
    \begin{split}
        \left| \mathcal{E}( {\rm emp}_{N}^{\epsilon} - \mu^{V}_{\infty}) - \mathcal{E}(\varphi - \mu^{V}_{\infty}) \right| &=  \left| \mathcal{G}( {\rm emp}_{N}^{\epsilon} -\varphi , {\rm emp}_{N}^{\epsilon} +\varphi - 2 \mu^{V}_{\infty} ) \right|\\
        \leq 4 \left\| h^{{\rm emp}_{N}^{\epsilon} -\varphi} \right\|_{L^{\infty}}. 
    \end{split}
\end{equation}

The procedure to bound $\left\| h^{{\rm emp}_{N}^{\epsilon} -\varphi} \right\|_{L^{\infty}}$ is similar to the last substep: in the case of assumption 5 a) of Definition \ref{def:admissible}, we take $x \in \mathbf{T}^{d}$ and $R>0$ and write
\begin{equation}
    \begin{split}
        &\int_{\mathbf{T}^{d}} g(x-y) ({\rm emp}_{N}^{\epsilon} -\varphi)(y) \, \mathrm d y \\
        =& \int_{B(x, R)}  g(x-y) ({\rm emp}_{N}^{\epsilon} -\varphi)(y) \, \mathrm d y  + \int_{\Sigma \setminus B(x, R)}  g(x-y) ({\rm emp}_{N}^{\epsilon} -\varphi)(y) \, \mathrm d y  \\
        \leq & C \left| \varphi \right|_{L^{\infty}} \left( \int_{B(0, R + \overline{\eta})} \frac{1}{|x|^{s}} + \int_{\Sigma \setminus B(0, R)} \frac{\overline{\eta}}{|x|^{s +1}} \right)\\
        \leq & C \left| \varphi \right|_{L^{\infty}} \left( \left( R+\overline{\eta} \right)^{ d - s} + \overline{\eta} R^{d - s -1}  \right).
    \end{split}
\end{equation}

The optimal $R$ is given by $R = \overline{\eta}$ and the result is that
\begin{equation}
    \left| \mathcal{E}( {\rm emp}_{N}^{\epsilon} - \mu^{V}_{\infty}) - \mathcal{E}(\varphi - \mu^{V}_{\infty}) \right| \leq C \left| \varphi \right|_{L^{\infty}} \overline{\eta}^{d - s}. 
\end{equation}

In the case of assumption 5 b) of Definition \ref{def:admissible},the result is that
\begin{equation}
    \left| \mathcal{E}( {\rm emp}_{N}^{\epsilon} - \mu^{V}_{\infty}) - \mathcal{E}(\varphi - \mu^{V}_{\infty}) \right| \leq C \left| \varphi \right|_{L^{\infty}} \overline{\eta}^{d} \left| \log \overline{\eta} \right|. 
\end{equation}

\textit{Substep 4.4:} Conclusion

Since the construction requires that $ N^{-\frac{1}{d}} \ll \overline{\eta} $, we conclude that for any $p > \frac{ s - d}{d}$,
\begin{equation}
     \left| {\rm F}_{N}(X_{N}, \mu^{V}_{\infty}) - \mathcal{E}(\varphi - \mu^{V}_{\infty}) \right| \leq N^{p},
\end{equation}
where $s$ is taken to be $0$ in the case of assumption 5 b) of Definition \ref{def:admissible}.
\end{proof}

\section{Acknowledgements}

We acknowledge the support from the German Research Foundation (DFG) via the research unit FOR 3013 “Vector- and tensor-valued surface PDEs” (grant no. NE2138/3-1). We are also very grateful to David Garcia-Zelada for conversations that were necessary for the completion of this work. 

\bibliographystyle{plain}
\bibliography{bibliography.bib}

\begin{thebibliography}{10}

\bibitem{armstrong2021local}
Scott Armstrong and Sylvia Serfaty.
\newblock Local laws and rigidity for coulomb gases at any temperature.
\newblock {\em The Annals of Probability}, 49(1):46--121, 2021.

\bibitem{armstrong2022thermal}
Scott Armstrong and Sylvia Serfaty.
\newblock Thermal approximation of the equilibrium measure and obstacle
  problem.
\newblock {\em Annales de la Facult{\'e} des sciences de Toulouse:
  Math{\'e}matiques}, 31(4):1085--1110, 2022.

\bibitem{bauerschmidt2019two}
Roland Bauerschmidt, Paul Bourgade, Miika Nikula, and Horng-Tzer Yau.
\newblock The two-dimensional coulomb plasma: quasi-free approximation and
  central limit theorem.
\newblock {\em Advances in Theoretical and Mathematical Physics},
  23(4):841--1002, 2019.

\bibitem{bekerman2018clt}
Florent Bekerman, Thomas Lebl{\'e}, and Sylvia Serfaty.
\newblock Clt for fluctuations of $\beta$-ensembles with general potential.
\newblock {\em Electronic Journal of Probability}, 23:1--31, 2018.

\bibitem{berman2014determinantal}
Robert~J Berman.
\newblock Determinantal point processes and fermions on complex manifolds:
  large deviations and bosonization.
\newblock {\em Communications in Mathematical Physics}, 327(1):1--47, 2014.

\bibitem{berman2019sharp}
Robert~J Berman.
\newblock Sharp deviation inequalities for the 2d coulomb gas and quantum hall
  states, i.
\newblock {\em arXiv preprint arXiv:1906.08529}, 2019.

\bibitem{bourgade2012bulk}
Paul Bourgade, L{\'a}szl{\'o} Erd{\H{o}}s, and Horng-Tzer Yau.
\newblock Bulk universality of general $\beta$-ensembles with non-convex
  potential.
\newblock {\em Journal of mathematical physics}, 53(9):095221, 2012.

\bibitem{bourgade2014universality}
Paul Bourgade, L{\'a}szl{\'o} Erd{\H{o}}s, and Horng-Tzer Yau.
\newblock Universality of general $\beta$-ensembles.
\newblock {\em Duke Mathematical Journal}, 163(6):1127--1190, 2014.

\bibitem{bourgade2014local}
Paul Bourgade, Horng-Tzer Yau, and Jun Yin.
\newblock Local circular law for random matrices.
\newblock {\em Probability Theory and Related Fields}, 159(3):545--595, 2014.

\bibitem{boursier2021optimal}
Jeanne Boursier.
\newblock Optimal local laws and clt for the circular riesz gas.
\newblock {\em arXiv preprint arXiv:2112.05881}, 2021.

\bibitem{chafai2014first}
Djalil Chafa{\"\i}, Nathael Gozlan, and Pierre-Andr{\'e} Zitt.
\newblock First-order global asymptotics for confined particles with singular
  pair repulsion.
\newblock {\em The Annals of Applied Probability}, 24(6):2371--2413, 2014.

\bibitem{chafai2018concentration}
Djalil Chafa{\"\i}, Adrien Hardy, and Myl{\`e}ne Ma{\"\i}da.
\newblock Concentration for coulomb gases and coulomb transport inequalities.
\newblock {\em Journal of Functional Analysis}, 275(6):1447--1483, 2018.

\bibitem{garcia2019concentration}
David Garc{\'\i}a-Zelada.
\newblock Concentration for coulomb gases on compact manifolds.
\newblock {\em Electronic Communications in Probability}, 24:1--18, 2019.

\bibitem{garcia2019large}
David Garc{\'\i}a-Zelada.
\newblock A large deviation principle for empirical measures on polish spaces:
  Application to singular gibbs measures on manifolds.
\newblock In {\em Annales de l'Institut Henri Poincar{\'e}, Probabilit{\'e}s et
  Statistiques}, volume~55, pages 1377--1401. Institut Henri Poincar{\'e},
  2019.

\bibitem{garcia2022generalized}
David Garc{\'\i}a-Zelada and David Padilla-Garza.
\newblock Generalized transport inequalities and concentration bounds for
  riesz-type gases.
\newblock {\em arXiv preprint arXiv:2209.00587}, 2022.

\bibitem{georgii2011gibbs}
Hans-Otto Georgii.
\newblock Gibbs measures and phase transitions.
\newblock In {\em Gibbs Measures and Phase Transitions}. de Gruyter, 2011.

\bibitem{hardy2021clt}
Adrien Hardy and Gaultier Lambert.
\newblock Clt for circular beta-ensembles at high temperature.
\newblock {\em Journal of Functional Analysis}, 280(7):108869, 2021.

\bibitem{johansson1998fluctuations}
Kurt Johansson.
\newblock On fluctuations of eigenvalues of random hermitian matrices.
\newblock {\em Duke mathematical journal}, 91(1):151--204, 1998.

\bibitem{lambert2021mesoscopic}
Gaultier Lambert.
\newblock Mesoscopic central limit theorem for the circular $\beta$-ensembles
  and applications.
\newblock {\em Electronic Journal of Probability}, 26:1--33, 2021.

\bibitem{lambert2021poisson}
Gaultier Lambert.
\newblock Poisson statistics for gibbs measures at high temperature.
\newblock In {\em Annales de l'Institut Henri Poincar{\'e}, Probabilit{\'e}s et
  Statistiques}, volume~57, pages 326--350. Institut Henri Poincar{\'e}, 2021.

\bibitem{lambert2019quantitative}
Gaultier Lambert, Michel Ledoux, and Christian Webb.
\newblock Quantitative normal approximation of linear statistics of
  $\beta$-ensembles.
\newblock {\em The Annals of Probability}, 47(5):2619--2685, 2019.

\bibitem{leble2017large}
Thomas Lebl{\'e} and Sylvia Serfaty.
\newblock Large deviation principle for empirical fields of log and riesz
  gases.
\newblock {\em Inventiones mathematicae}, 210(3):645--757, 2017.

\bibitem{leble2018fluctuations}
Thomas Lebl{\'e} and Sylvia Serfaty.
\newblock Fluctuations of two dimensional coulomb gases.
\newblock {\em Geometric and Functional Analysis}, 28(2):443--508, 2018.

\bibitem{maida2014free}
Myl{\`e}ne Ma{\"\i}da and {\'E}douard Maurel-Segala.
\newblock Free transport-entropy inequalities for non-convex potentials and
  application to concentration for random matrices.
\newblock {\em Probability Theory and Related Fields}, 159(1-2):329--356, 2014.

\bibitem{nguyen2022mean}
Quoc~Hung Nguyen, Matthew Rosenzweig, and Sylvia Serfaty.
\newblock Mean-field limits of riesz-type singular flows.
\newblock {\em Ars Inveniendi Analytica}, 2022.

\bibitem{padilla2022large}
David Padilla-Garza.
\newblock Large deviations principle for the tagged empirical field of a
  general interacting gas.
\newblock {\em arXiv preprint arXiv:2208.01516}, 2022.

\bibitem{padilla2023concentration}
David Padilla-Garza.
\newblock Concentration inequality around the thermal equilibrium measure of
  coulomb gases.
\newblock {\em Journal of Functional Analysis}, 284(1):109733, 2023.

\bibitem{peilen2024local}
Luke Peilen.
\newblock Local laws and a mesoscopic clt for $\beta$-ensembles.
\newblock {\em Communications on Pure and Applied Mathematics}, 2024.

\bibitem{petrache2017next}
Mircea Petrache and Sylvia Serfaty.
\newblock Next order asymptotics and renormalized energy for riesz
  interactions.
\newblock {\em Journal of the Institute of Mathematics of Jussieu},
  16(3):501--569, 2017.

\bibitem{rosenzweig2023global}
Matthew Rosenzweig and Sylvia Serfaty.
\newblock Global-in-time mean-field convergence for singular riesz-type
  diffusive flows.
\newblock {\em The Annals of Applied Probability}, 33(2):954--998, 2023.

\bibitem{rotskoff2018parameters}
Grant Rotskoff and Eric Vanden-Eijnden.
\newblock Parameters as interacting particles: long time convergence and
  asymptotic error scaling of neural networks.
\newblock {\em Advances in neural information processing systems}, 31, 2018.

\bibitem{rotskoff2022trainability}
Grant Rotskoff and Eric Vanden-Eijnden.
\newblock Trainability and accuracy of artificial neural networks: An
  interacting particle system approach.
\newblock {\em Communications on Pure and Applied Mathematics},
  75(9):1889--1935, 2022.

\bibitem{rotskoff2018neural}
Grant~M Rotskoff and Eric Vanden-Eijnden.
\newblock Neural networks as interacting particle systems: Asymptotic convexity
  of the loss landscape and universal scaling of the approximation error.
\newblock {\em stat}, 1050:22, 2018.

\bibitem{rougerie2016higher}
Nicolas Rougerie and Sylvia Serfaty.
\newblock Higher-dimensional coulomb gases and renormalized energy functionals.
\newblock {\em Communications on Pure and Applied Mathematics}, 69(3):519--605,
  2016.

\bibitem{rougerie2015incompressibility}
Nicolas Rougerie and Jakob Yngvason.
\newblock Incompressibility estimates for the laughlin phase.
\newblock {\em Communications in Mathematical Physics}, 336(3):1109--1140,
  2015.

\bibitem{ruelle1967variational}
David Ruelle.
\newblock A variational formulation of equilibrium statistical mechanics and
  the gibbs phase rule.
\newblock {\em Communications in Mathematical Physics}, 5:324--329, 1967.

\bibitem{ruelle1968statistical}
David Ruelle.
\newblock Statistical mechanics of a one-dimensional lattice gas.
\newblock {\em Communications in Mathematical Physics}, 9:267--278, 1968.

\bibitem{sandier20152d}
Etienne Sandier and Sylvia Serfaty.
\newblock 2d coulomb gases and the renormalized energy.
\newblock {\em The Annals of Probability}, 43(4):2026--2083, 2015.

\bibitem{serfaty2015coulomb}
Sylvia Serfaty.
\newblock {\em Coulomb gases and Ginzburg--Landau vortices}.
\newblock 2015.

\bibitem{serfaty2017microscopic}
Sylvia Serfaty.
\newblock Microscopic description of log and coulomb gases.
\newblock {\em arXiv preprint arXiv:1709.04089}, 2017.

\bibitem{serfaty2023gaussian}
Sylvia Serfaty.
\newblock Gaussian fluctuations and free energy expansion for coulomb gases at
  any temperature.
\newblock In {\em Annales de l'Institut Henri Poincare (B) Probabilites et
  statistiques}, volume~59, pages 1074--1142. Institut Henri Poincar{\'e},
  2023.

\bibitem{wen2024coupling}
Yuxiao Wen, Eric Vanden-Eijnden, and Benjamin Peherstorfer.
\newblock Coupling parameter and particle dynamics for adaptive sampling in
  neural {G}alerkin schemes.
\newblock {\em Physica D: Nonlinear Phenomena}, page 134129, 2024.

\end{thebibliography}

\end{document}